\documentclass{article}
\usepackage{amsfonts}
\usepackage{amsmath}
\usepackage{enumerate}
\usepackage{tikz}
\usetikzlibrary{matrix,arrows,positioning,decorations.markings,decorations.pathmorphing}
\newtheorem{theorem}{Theorem}

\newtheorem{definition}[theorem]{Definition}

\newtheorem{lemma}[theorem]{Lemma}

\newtheorem{proposition}[theorem]{Proposition}
\newtheorem{remark}[theorem]{Remark}

\newenvironment{proof}[1][Proof]{\noindent\textbf{#1.} }{\ \rule{0.5em}{0.5em}}

\begin{document}

\title{Reduction (by stages) in the whole Lagrange--Poincar\'{e} category}
\author{Miguel \'{A}ngel Berbel \and Marco Castrill\'{o}n L\'{o}pez}
\date{}
\maketitle

\begin{abstract}
We complete the reduction scheme in the whole $\mathfrak{LP}$ category, introduced in \cite{CMR} to perform Lagrangian reduction by stages. We answer affirmatively the open question of whether reduction can be done in the whole category and analyze the Noether theorem on $\mathfrak{LP}$-bundles, the relationship with Hamiltonian reduction by stages and the geometric aspects of the definition of this category.
\end{abstract}

\section{Introduction}
Symmetries are a central topic in Mechanics. They provide tools to simplify the dynamics of the problem under study and they are associated to conservation laws. In this sense, the geometric outlook provided by Geometric Mechanics has been successfully exploited to build a framework within the language of actions of Lie groups on manifold (just to mention some essential pieces of literature, we refer the reader to \cite{abraham, arnold, MRbook} as well as their references). The key point is the introduction of the quotient of the phase spaces of the dynamical problems, a picture that is commonly known as \emph{reduction}. For example, in the Lagrangian formulation of Mechanics, we consider a Lie group $G$ of symmetries acting (freely and properly) on the configuration manifold $Q$ so that a reduced Lagrangian and variational principle is defined in the the quotient $(TQ)/G$ under the lift of the action on the tangent bundle. A similar situation is defined on the Hamiltonian side of the picture, where the Poisson structure and the Hamiltonian are projected to $T^*Q/G$.

There are many occasions where the Lie group $G$ of symmetries of the system contains transformations of remarkably distinct nature. More precisely, in some occasions, there is a normal subgroup of symmetries $N$ encoding one type of information different (in terms of physical or mathematical reasons) to the remaining symmetries given in $G/N$. The action of the global symmetry group is split and the reduction process can be organized as a concatenation of simpler steps, first by $N$ and then by the quotient $G/N$. A prototypical example is present in the geometric models of underwater vehicles with rotors (see \cite{LM}). Nonetheless, since the quotient of a tangent or cotangent bundle is not necessarily a bundle of that type, the iteration of the reduction process implies reducing Lagrangians or Hamiltonians defined in a different kind of phase spaces, included in a wider category of phase spaces that, in particular, contains tangent or cotangent bundles.

This problem was tackled (for the Lagrangian side) in the celebrated article \cite{CMR}, where the authors introduced a new category $\mathfrak{LP}$ of Lagrange--Poincar\'e bundles that is stable under reduction of actions of Lie groups. In this situation, we can perform reduction several times and, if we reduce by a $N$ and afterward by $K=G/N$, the final result is equivalent to a direct reduction by $G$, provided some auxiliary connections used along the process are conveniently chosen. Apparently, the double reduction seems to be a longer way but, with it, one keeps control of the double nature of the symmetries mentioned above, which can be specially useful in the study of stability or conservation laws, for instance. Since then, many articles have followed and applied that seminal work (for example, some few references are \cite{Bl}, \cite{Fer}, \cite{Gay}, \cite{GS}). 
 
However, the work in \cite{CMR} restricted its main results to the subcategory $\mathfrak{RI}$ of Lagrange--Poincar\'e bundles coming from reducing an initial tangent bundle, so that the question whether the reduction procedure could be done in the whole category was left as an open problem. In this article, we both answer affirmatively to that question and complete some other aspects of this theory.

In this paper, we complete the reduction scheme in the whole $\mathfrak{LP}$ category. In addition, we provide some additional geometrical insight about the elementes of the category and the structure of their variational equations. On one hand, we provide a characterization of the Lie bracket on the sections of $\mathfrak{LP}$-bundle in geometric terms. This is particularly useful to understand the naturality of that bracket. On the other hand, the Noether theorem on $\mathfrak{LP}$-bundles is analyzed. Here, the notion of the conservation law is translated into a drift equation, a fact consistent with other theories with non-holonomic constraints, but with supplementary property that the drift equations fits into the variational equations of the reduced $\mathfrak{LP}$-bundle when reduction is performed (see a similar situation in Field Theories in \cite{marco}). We also go deeper in the relationship between reduction in the Lagrange-Poincar\'e category and Poisson reduction with a more detailed analysis of the relationship of the $\mathfrak{LP}$ Lie bracket with the one in Hamiltonian reduction by stages.

We complete the work with some examples. Even though \cite{CMR} already contains an indication of examples outside $\mathfrak{RI}$ (in terms of non-orientable manifolds), we provide a bigger class by using non-trivial flat bundles. As we know, these bundles play an important role in many geometrical situations. Choosing a convenient Lagrangian, we are able to include dynamics as simple as parallel transport in the purely $\mathfrak{LP}$ developed in the article. Finally, we also present an alternative way of studying systems depending on parameters to the one presented, for example, in \cite{CM87,HMR98A}.

\section{Preliminaries}

\subsection{Fiber bundles} \label{req-bundles}

Let $G$ be a Lie group acting freely and properly on a manifold $Q$. The action $\rho:G\times Q\to Q$ will be assumed to be on the left, although all the results in this article can be stated in a similar way for right actions. The natural projection $\pi :Q\rightarrow Q/G$ is a principal $G$-bundle such that the
fibers are exactly the orbits of the action. Therefore, if we denote by
\begin{equation*}
V_{q}Q=\{v\in T_{q}Q|d_q\pi (v)=0\},\qquad q\in Q,  \label{verticales}
\end{equation*}
the vertical tangent vectors, there is a natural identification between the vertical subbundle $VQ\subset TQ$ and $Q\times
\mathfrak{g}$ given by
$$
Q\times \mathfrak{g}\ni (q,\xi )\mapsto \xi ^{Q}_{q}=\left. \frac{d}{dt}%
\right\vert _{t=0}\exp (t\xi )\cdot q\in TQ,
$$
where the dot stands for the action and $\mathfrak{g}$ is the Lie algebra of
$G$. We denote by $\tilde{\mathfrak{g}}\rightarrow Q/G$ the adjoint bundle to $Q\to Q/G$, that is, the associated bundle $(Q\times \mathfrak{g})/G$ by the adjoint action of $G$ on $\mathfrak{g}$. Its elements will be denoted by $[q,\xi]_{G}$, $q\in Q$, $\xi \in \mathfrak{g}$. We recall that $\tilde{\mathfrak{g}
}\rightarrow Q/G$ is a Lie algebra bundle, that is, a vector bundle equipped with a fiberwise Lie algebra given by
$$\lbrack \lbrack q,\xi _{1}]_{G},[q,\xi _{2}]_{G}]=[q,[\xi _{1},\xi
_{2}]]_{G},\qquad \lbrack q,\xi _{1}]_{G},[q,\xi _{2}]_{G}\in
\tilde{\mathfrak{g}}_{x},x=\pi (q).$$
There is a bijection between sections $\bar{\xi}:Q/G\rightarrow \tilde{\mathfrak{g}}$ of the adjoint bundle and $\pi $-vertical $G$-invariant
vector fields $X_{\bar{\xi}}$ on $Q$ by setting $(X_{\bar{\xi}})_{q}=\xi
_{q}^{Q}$ for any $q\in Q$ where $\bar{\xi}(x)=[q,\xi ]_{G}$ and $x=\pi (q)$.

A principal connection $A$ on $Q\rightarrow Q/G$ is a $\mathfrak{g}$-valued
$1$-form on $Q$ such that $A(\xi _{q}^{Q})=\xi $, for any $\xi \in
\mathfrak{g}$, $q\in Q$, and $\rho_{g}^{\ast }A=\mathrm{Ad}_{g}\circ A$,
where $\rho_{g}:Q\to Q$ denotes the (left) action by $g\in G$. The form $A$
splits the tangent spaces as $T_{q}Q=H_{q}Q\oplus V_{q}Q$, for all $q\in Q$, where
$$H_{q}Q=\ker A_{q}=\{v\in T_{q}Q|A_{q}(v)=0\},\qquad q\in Q,$$
is called the horizontal subspace. The collection of all $H_qQ$, constitute a distribution $HQ$. Note that each $H_{q}Q$ can be identified with the tangent space, $T_x(Q/G)$, $x=\pi(q)$, through $d_q\pi$.

The curvature of a connection $A$ is a $\mathfrak{g}$-valued $2$-form given by $$B(v,w)=dA(\mathrm{Hor}(v),\mathrm{Hor}(w)),$$ where $v,w\in T_qQ$ and $\mathrm{Hor}(v)$ is the horizontal part of $v$ according to the decomposition $T_{q}Q=H_{q}Q\oplus V_{q}Q$. Intuitively, the curvature is the obstruction to the Frobenius integrability of  $HQ$. We note that the curvature can be also regarded as a 2-form on $Q/G$ with values in the adjoint bundle as
$$ \tilde{B}(X,Y)=[q,B(X^h,Y^h)]_G,$$ for any $X,Y\in T_x(Q/G)$, where $X^h,Y^h \in H_qQ$, is the horizontal lift with respect to the connection to any point $q$ of the fiber of $x$.

Given a vector bundle $\tau:V\to Q$, an affine connection is a map
\begin{equation*}
\begin{split}
\nabla:\mathfrak{X}^{\infty}(Q)\times \Gamma(V) &\to \Gamma(V) \\
(X,s) &\mapsto \nabla_Xs,
\end{split}
\end{equation*}
which is $\mathcal{C}^{\infty}(Q)$-linear in the first entry and for all $s_1,s_2\in \Gamma(V)$, $X \in \mathfrak{X}^{\infty}(Q)$ and $f_1,f_2 \in \mathcal{C}^{\infty}(Q)$ satisfies,
$$\nabla_{X}(f_1s_1+f_2s_2)=X[f_1]s_1+f_1\nabla_{X}s_1+X[f_2]s_2+f_2\nabla_{X}s_2,$$
where $X[f]$ denotes the derivative of $f$ in the direction of  the field $X$. An affine connection $\nabla$ splits $T_vV=H_vV\oplus V_vV$, $v\in V$, with $H_vV=ds(T_q Q)$ for a (local) section $s:Q\to V$ such that $(\nabla _X s)_q=0$ for all $X\in T_q Q$. Furthermore, an affine connection defines a covariant derivative for curves $v(t)\in V$ as
 $$\frac{Dv(t)}{dt}=\nabla _{\dot{q}(t)}s,$$ where $q(t)=(\tau \circ v)(t)$ and $s$ is a local section around a neighborhood of $q(t)$ such that $v(t)=s(q(t))$. It can be proved that there is a bijection between affine connections and covariant derivatives on $V$.

%Derivada de formas------------------------------------
We denote by $\Omega^p(Q,V)=\Gamma(\bigwedge^p T^*Q\otimes V)$ the space of $V$-valued $p$-differential forms on $Q$. Unlike real-valued differential forms on $Q$, the space
\[
\Omega (Q,V)=\bigoplus_p\Omega^p(Q,V),
\]
of $V$-valued differential forms on $Q$ is not an algebra with respect to the wedge product. Yet, it is a $\Omega (Q)$-module with the following wedge product
$$\omega\wedge\eta(X_1,\dots,X_{p+q})=\sum_\sigma \mathrm{sign}(\sigma)\omega(X_{\sigma(1)},\dots,X_{\sigma(p)})\eta(X_{\sigma(p+1)},\dots, X_{\sigma(p+q)}),$$
where $\omega\in \Omega^p (Q),\eta\in \Omega^q (Q,V)$ and the sum is over all permutations $\sigma$ of $p+q$ elements.

Given an affine connection $\nabla$, there is a unique operator
\[
d_{\nabla}:\Omega^{p}(Q,V) \to \Omega^{p+1}(Q,V),
\]
such that $d_{\nabla}$ acting on $s\in\Omega^{0}(Q,V)=\Gamma(V)$ coincides with $\nabla s\in \Omega^{1}(Q,V)$, and satisfies the Leibniz identity $d_{\nabla}(\omega\wedge\eta)=(d\omega)\wedge\eta+(-1)^{\mathrm{deg}(\omega)}\omega\wedge(d_{\nabla}\eta),\omega\in \Omega^{\bullet} (Q),\eta\in \Omega^q (Q,V)$.
In fact, there is a one to one correspondence between affine connections and linear operators $d_{\nabla}$ defined on
$\Omega^{\bullet}(Q,V)$, increasing the degree by one and satisfying the Leibniz identity. As usual, there is a an intrinsic formula given as
\begin{align*}
d_\nabla(\eta)(X_1,\dots,X_{q+1})=&\sum_{i<j}(-1)^{(i+j)}\eta([X_i,X_j],X_1,\dots,\hat{X}_i,\dots,\hat{X}_j,\dots,X_{q+1})\\&+\sum_{i=1}^{q+1}(-1)^{(i+1)}\nabla_{X_i}(\eta(X_1,\dots,\hat{X}_i,\dots,X_{q+1})),
\end{align*}
where $\hat{X}_i$ means that $X_i$ is omitted. In general, $d_{\nabla}^2=d_{\nabla}\circ d_{\nabla}\neq 0$. The curvature of an affine connection $k_{\nabla}\in \Omega^2(Q,\mathrm{End}(V))$, defined as
$$k_{\nabla}(X,Y)s=\nabla_X\nabla_Ys-\nabla_Y\nabla_Xs-\nabla_{[X,Y]}s$$ for all $X,Y\in\mathfrak{X}(Q)$ and all $s\in\Gamma(V)$, can be seen as the failure to get that property since
$$d_{\nabla}^2(s)(X,Y)=k_{\nabla}(X,Y)(s).$$
%https://www.mathematik.hu-berlin.de/~wendl/pub/connections_chapter5_2.pdf Chris Wendl
%http://www.staff.science.uu.nl/~crain101/DG-2016/chapter1.pdf CRAINIC
%----------------------------------------------

Let $E=Q\times_G F$ be a bundle associated to the $G$-principal bundle $\pi:Q\to Q/G$ and defined by the action of $G$ on a manifold $F$. A principal connection  $A$ on $\pi:Q\to Q/G$ defines a horizontal distribution $HE\subset TE$ complementary to the vertical bundle
\[
T_{[q,f]_G}E=V_{[q,f]_G}E\oplus H_{[q,f]_G}E, \qquad [q,f]_G\in E,
\]
by the condition
\[
H_{[q,f]_G}E=d_q\psi_f(H_qQ),
\]
where $\psi_f: q\in Q\mapsto [q,f]_G\in E.$ In case $E$ is a vector bundle, the distribution hereby obtained gives rise to an affine connection on $E$, and hence, a covariant derivative.

Given a vector bundle $V\to Q$, an action $\rho$ of a Lie group $G$ on $V$ is called a vector bundle action if the maps $\rho_g:V\to V$, $g\in G$, are  vector bundle isomorphisms and the action induced in the base is free and proper. The rules $$[v_q]_G+[w_q]_G=[v_q+w_q] \text{ and } \lambda [v_q]_G=[\lambda v_q]_G,$$ where $[v_q]_G,[w_q]_G\in V/G$ stand for the equivalence classes of $v_q,w_q\in V_q$ and $\lambda \in \mathbb{R}$, define a vector bundle structure on $V/G$ over $Q/G$ with projection $\tau _G:V/G\to Q/G$ given by $\tau _G([v]_G)=[\tau (v)]_G$. The projection  $\pi_{V,G}:V\to V/G$ is a surjective vector bundle homomorphism covering $\pi:Q\to Q/G$
\begin{center}
\begin{tikzpicture}[scale=1.5] \label{quo.vect.ble}
\node (0) at (0,1) {$V$};
\node (A) at (1.5,1) {$V/G$};
\node (B) at (1.5,0) {$Q/G$};
\node (C) at (0,0) {$Q$};
\draw[->,font=\scriptsize,>=angle 90]
(0) edge node[above]{$\pi_{V,G}$} (A)
(0) edge node[right]{$\tau$} (C)
(C) edge node[above]{$\pi$} (B)
(A) edge node[right]{$\tau_G$} (B);
\end{tikzpicture}
\end{center}
whose restriction to each fiber is a linear isomorphism. There is a bijection between the space $\Gamma (V/G)$ of sections of $V/G$ and the space $\Gamma ^G (V)$ of $G$-invariant sections of $V$. If $\Gamma(V)$ has a Lie algebra structure invariant by the action, $\Gamma^G(V)$ is a Lie subalgebra which makes $\Gamma(V/G)$ a Lie algebra.

\subsection{Lagrange--Poincar\'e reduction} \label{usuallp}

The action defined by a Lagrangian function $L :TQ \to \mathbb{R}$ on the set $\Omega(Q;q_0,q_1)$ of curves $q(t)$ with endpoints $q_0,q_1\in Q$ is defined as
\[
S(q(t))=\int _{t_0}^{t_1} L(\dot{q}(t))dt,
\]
where $\dot{q}(t)$ denotes the natural lift of $q(t)$ to $TQ$. A curve $q(t)\in \Omega(Q;q_0,q_1)$ is said to be critical with respect to the action if for any smooth deformation $\{q_\lambda(t)\}_{\lambda\in \mathbb{R}}$ of curves in $\Omega(Q;q_0,q_1)$, with  $q_0(t)=q(t)$, we have
\[
\left. \frac{d}{d\lambda}\right|_{\lambda =0} S(q_\lambda (t))=0.
\]
The vector field $\delta q= d/d\lambda |_{\lambda =0} q_\lambda$ is called a variation, and the set of them is denoted by $\Delta_q$. Note that any vector field along $q(t)$ with vanishing endpoints is a variation. Critical curves satisfy the renowned Euler--Lagrange equations, which in standard coordinates $(q^1,\dots,q^n,\dot{q}^1,\dots,\dot{q}^n)$ on $TQ$ can be written as
\[
\frac{\partial L}{\partial q^i}(q,\dot{q})  -\frac{d}{dt}\left(\frac{\partial L}{\partial \dot{q}^i}(q,\dot{q})\right)=0, \qquad \forall i=1,\dots , n.
\]

It is sometimes useful to regard the tangent bundle $TQ$ as the space $J^1(\mathbb{R},Q)$, the jet bundle of classes of curves from $\mathbb{R}$ to $Q$. The class of a curve $\gamma(t)$ is denoted by $[\gamma(t)]^{(1)}$ and the curve $\dot{q}(t)$ in $TQ$ is just the curve $[q]^{(1)}(t)$.
%such that $[q]^{(1)}(a)=[q(a+t)]^{(1)}$.
Similarly, for $k\geq 2$, we understand the $k$-th order tangent bundle $T^{(k)}Q$ as the $k$-order jet bundle $J^k(\mathbb{R},Q)$. The lift of a curve $q(t)$ to $T^{(k)}Q$ is denoted by $[q]^{(k)}(t)$.
%The variational principle can be restated to work directly on $TQ$, where $L$ is defined. The space of curves considered is the space of lifted curves $l\Omega(Q;q_0,q_1)$ and their %variations are lifted variations, $\Delta^l_q$, in the following sense: If $\delta q$ is obtained from a deformation $q(t,\lambda)$, the lifted variation, $\delta \dot{q}$, is %obtained from the derivation $\dot{q}(t,\lambda)=[q_{\lambda}(t)]^{(1)}$.

Given an affine connection $\nabla$ on $TQ\to Q$, the Euler--Lagrange operator is the bundle map
\begin{align*}
\mathcal{EL}(L):T^{(2)}Q&\to T^*Q \\
[q]^{(2)}&\mapsto \frac{\partial L}{\partial q}(\dot{q}(t))-\frac{D}{dt}\left(\frac{\partial L}{\partial \dot{q}}(\dot{q}(t))\right),
\end{align*}
where $\partial L /\partial q$ and $\partial L /\partial \dot{q}$ are respectively the horizontal and vertical part of $\partial L /\partial[q]^{(1)}$ with respect to $\nabla$, and $\frac{D}{dt}$ is the covariant derivative defined by $\nabla$ on the dual bundle $T^*Q$. It is not hard to prove that the Euler--Lagrange operator is independent of the chosen $\nabla$. In addition, it provides an intrinsic description of the local Euler--Lagrange equations above so that a curve $q(t)\in \Omega (Q;q_0,q_1)$ is critical if and only if $\mathcal{EL}(L)([q]^{(2)})=0$.

We now consider a Lie group $G$ acting freely and properly on the configuration manifold $Q$ as well as the natural induced action on $TQ$. If we assume that the Lagrangian $L$ is invariant with respect to this action,it drops to a function
\[
l:TQ/G \to \mathbb{R}
\]
called the reduced Lagrangian. The unreduced variational problem can also be dropped to the quotient as follows. On one hand, the set of curves are $[\dot{q}]_G$, for $q\in \Omega(Q;q_0,q_1)$, or using jet notation $[q^{(1)}]_G$. On the other, the set of admissible variations, denoted $\Delta^{\ell}_q$, are the projection $[\delta\dot{q}]_G=[\delta q^{(1)}]_G$ of variations $\delta q \in \Delta_q$. This process is known as the Lagrange--Poincar\'e reduction. However, the projection of these curves and variations requires a better understanding of the bundle $(TQ)/G\to Q/G$. For any choice of a principal connection $A$ on $Q\to Q/G$ we define the map
\begin{equation*}
\begin{split}
\alpha_A:TQ/G& \to T(Q/G)\oplus \tilde{\mathfrak{g}} \\
[v_q]_G &\mapsto T\pi(v_q)\oplus [q,A(v_q)]_G. \\
\end{split}
\end{equation*}
This map is a vector bundle diffeomorphism. Given a curve $q\in \Omega(Q;q_0,q_1)$, we consider the curves
\[
\dot{x}(t)=T\pi(\dot{q}(t)),\qquad \bar{\xi}(t)=[q(t),A(\dot{q}(t))]_G,
\]
so that the admissible curves $[\dot{q}]_G(t)=[\dot{q}(t)]_G$ of the reduced variational problem on $TQ/G$ are identified with $\dot{x}(t)\oplus \bar{\xi}(t)$ via $\alpha_A$. As the notation suggests, $\dot{x}(t)$ coincides with the tangent lift of $x(t)=\pi(q(t))$. In other words, the set of curves for the variational problem of the reduced Lagrangian $l$ lies in
\[
\Omega(Q/G;x_0,x_1)\oplus\Omega(\tilde{\mathfrak{g}};x_0,x_1),
\]
where $x_0=\pi(q_0), x_1=\pi(q_1)$ and $\Omega(\tilde{\mathfrak{g}};x_0,x_1)$ is the space of curves in $\tilde{\mathfrak{g}}$ with endpoints whose projections are $x_0,x_1$. In fact, the curve in $\Omega (\tilde{\mathfrak{g}};x_0,x_1)$ is the main object, since the projection of $\bar{\xi}(t)$ to $Q/G$ is $x(t)$. We can thus say that the set of admissible reduced curves is exactly $\Omega(\tilde{\mathfrak{g}};x_0,x_1)$. However, we will keep below the notation $x(t)\oplus \bar{\xi}(t)$ for the admissible curves to keep track of the curve in the reduced configuration manifold $Q/G$.

With respect to the set of admissible variations, we have that the first factor of the projection of $\delta q ^{(1)}$ is simply $\delta x^{(1)}$, the lift of the variation of $x(t)$. However, the induced variation of $\bar{\xi}(t)$ is more involved. For that, one studies the cases where $\delta q$ is vertical or horizontal with respect to the projection $\pi :Q\to Q/G$ and the connection. One can prove that (see \cite{CMR}), when $\delta q$ is vertical, that is, $\delta q(t)=\eta (t)^Q_{q(t)}$ with $\eta:I\to \mathfrak{g}$ and $\eta (t_0)=\eta (t_1)=0$, then $\delta \bar{\xi}$ is given by
\[
\delta \bar{\xi}(t)=\frac{D}{dt}\bar{\eta}(t)+[\bar{\xi}(t), \bar{\eta}(t)]_G,
\]
where $\bar{\eta}(t)=[q(t),\eta(t)]_G$. For horizontal variations $\delta q$, that is, $A(\delta q)=0$, the variation $\delta\bar{\xi}$ is given by
$$\delta \bar{\xi}(t)=\delta x(t)^h_{\bar{\xi}(t)}+\tilde{B}_x(\delta x(t), \dot{x}(t)).$$
It is worth noting that, despite $\delta q$ being horizontal, $\delta\bar{\xi}$ is not horizontal, since there is a vertical summand coming from the curvature of the connection.

As any variation $\delta q$ can be decomposed as the sum of a vertical and horizontal components, we can conclude that the set of admissible variations of the reduced variational problem for a curve $x(t)\oplus \bar{\xi}(t)\in \Omega(Q/G;x_0,x_1)\oplus\Omega(\tilde{\mathfrak{g}};x_0,x_1)$ are
$$\delta\dot{x}\oplus \left( \frac{D}{dt}\bar{\eta}+[\bar{\xi}, \bar{\eta}]_G+ \tilde{B}_x(\delta x, \dot{x})\right) $$
where $\delta \dot{x}$ is the lift of a free variation $\delta x\in \Delta_x$ and $\bar{\eta}(t)$ is a curve in $\tilde{\mathfrak{g}}$ such that $\tau(\bar{\eta}(t))=x(t)$ and $\bar{\eta}(t_0)=\bar{\eta}(t_1)=0$.

Once we understand both the set of admissible curves and variations, the criticality of curves for the reduced variational principle can be written by means of the vanishing of a bundle morphism, called the Lagrange--Poincar\'e morphism,
\[
\mathcal{LP}(l):T^{(2)}Q/G\simeq T^{(2)}(Q/G)\oplus 2\tilde{\mathfrak{g}}\to T^*(Q/G)\oplus\tilde{\mathfrak{g}}^*,
\]
which can be split in two components as
\[
\mathcal{LP}(l)=\mathrm{Hor}(\mathcal{LP})(l)\oplus\mathrm{Ver}(\mathcal{LP})(l):T^{(2)}(Q/G)\oplus 2\tilde{\mathfrak{g}}\to T^*(Q/G)\oplus\tilde{\mathfrak{g}}^*,
\]
with
\begin{align*}
\mathrm{Hor}(\mathcal{LP})(l):T^{(2)}(Q/G)\oplus 2\tilde{\mathfrak{g}}&\to T^*(Q/G)\\
[x]^{(2)}\oplus [\bar{\xi}]^{(1)}&\mapsto \frac{\partial l}{\partial x}(\dot{x},\bar{\xi})-\frac{D}{dt}\frac{\partial l}{\partial \dot{x}}(\dot{x},\bar{\xi})-\left\langle\frac{\partial l}{\partial \bar{\xi}}(\dot{x},\bar{\xi}),\tilde{B}_q(\dot{x},\cdot)\right\rangle,\\
\mathrm{Ver}(\mathcal{LP})(l):T^{(2)}(Q/G)\oplus 2\tilde{\mathfrak{g}}&\to\tilde{\mathfrak{g}}^* \\
[x]^{(2)}\oplus [\bar{\xi}]^{(1)}&\mapsto  \text{ad}^*_{\bar{\xi}}\frac{\partial l}{\partial \bar{\xi}}(\dot{x},\bar{\xi})-\frac{D}{dt}\frac{\partial l}{\partial \bar{\xi}}(\dot{x},\bar{\xi}) .
\end{align*}

%---------------------------------
\subsection{Noether Current and Vertical Equations}\label{noetherTQ}
Given a Lagrangian $L:TQ\to \mathbb{R}$ invariant under the action of a Lie group $G$, we define the Noether current as the map
\begin{align*}
J:TQ\to &\mathfrak{g}^* \\
\dot{q}\mapsto &J(\dot{q})(\eta)=\left\langle \frac{\partial L}{\partial \dot{q}}(\dot{q}),\eta_q^Q\right\rangle .
\end{align*}
The Noether Theorem states that a solution $q(t)$ of the Euler--Lagrange equations of $L$ satisfies $\left\langle dJ(\dot{q}(t))/dt,\eta\right\rangle =0$ for all $\eta\in\mathfrak{g}$. Since,
\begin{align*}
J(g\dot{q})(\mathrm{Ad}_g\eta)&=\left\langle  \frac{\partial L}{\partial \dot{q}}(g\dot{q}(t)),(\mathrm{Ad}_g\eta)_{gq}^Q\right\rangle =\left.\frac{d }{ds}\right|_{s =0} L(g\dot{q}+s\left.\frac{d}{d\lambda}\right|_{\lambda =0} g\exp(\lambda\eta)g^{-1}\cdot g\dot{q})\\ &=\left.\frac{d }{ds}\right|_{s =0} L(\dot{q}+g^{-1}s\left.\frac{d}{d\lambda}\right|_{\lambda =0}g\exp(\lambda\eta)g^{-1}\cdot g\dot{q}) \\ &=\left.\frac{d }{ds}\right|_{s =0} L(\dot{q}+s\eta_q^Q)  =\left\langle  \frac{\partial L}{\partial \dot{q}}(\dot{q}(t)),\eta_q^Q\right\rangle =J(\dot{q})(\eta) ,
\end{align*}
$J$ is $G$-equivariant with respect to the lifted action on $TQ$ and the co-adjoint action on $\mathfrak{g}^*$. Then,
\begin{align*}
\tilde{J}:TQ\times \mathfrak{g}\to & \mathbb{R}\\
(\dot{q},\eta) \mapsto &\tilde{J}(\dot{q},\eta)=\left\langle  \frac{\partial L}{\partial \dot{q}}(\dot{q}(t)),\eta_q^Q\right\rangle .
\end{align*}
is a $G$-invariant real function which can be reduced to the quotient bundle,
$(TQ\times \mathfrak{g})/G\cong (TQ)/G\oplus \tilde{\mathfrak{g}}\cong T(Q/G)\oplus \tilde{\mathfrak{g}}\oplus \tilde{\mathfrak{g}},$ as follows
\begin{align*}
\tilde{j}:T(Q/G)\oplus \tilde{\mathfrak{g}}\oplus \tilde{\mathfrak{g}} \to & \mathbb{R}\\
(\dot{x},\bar{\xi},\bar{\eta}) \mapsto &\tilde{J}(\dot{x}^h_q+\xi_q^Q,\eta).
\end{align*}
This reduced function coincides with the vertical derivative of the reduced Lagrangian in the sense that
\begin{align*}
\tilde{j}(\dot{x},\bar{\xi},\bar{\eta})&=\tilde{J}(\dot{x}^h_q+\xi_q^Q,\eta)=\left.\frac{d}{ds}\right|_{s =0} L(\dot{x}^h_q+\xi_q^Q+s\eta_q^Q)\\ &=\left.\frac{d}{ds}\right|_{s =0} l(\dot{x}\oplus\bar{\xi}+s(0\oplus\bar{\eta}) =\left\langle \frac{\partial l}{\partial \bar{\xi}}, \bar{\eta}\right\rangle.
\end{align*}
As a consequence,
\begin{align*}
\left\langle \frac{d}{dt}J(\dot{q}(t)),\eta\right\rangle &=  \frac{d}{dt} \left\langle J(\dot{q}(t)),\eta\right\rangle =\frac{d}{dt} \tilde{J}(\dot{q}(t),\eta)\\
&=\frac{d}{dt} \tilde{j}(\dot{x}(t),\bar{\xi}(t),\bar{\eta}(t))=\frac{d}{dt}\left\langle \frac{\partial l}{\partial \bar{\xi}}(\dot{x}(t)\oplus\bar{\xi}(t)), \bar{\eta}(t)\right\rangle.
\end{align*}
From the definition of covariant derivative on a dual bundle
\begin{align*}
\frac{d}{dt}&\left\langle \frac{\partial l}{\partial \bar{\xi}}(\dot{x}(t)\oplus\bar{\xi}(t)), \bar{\eta}(t)\right\rangle\\ &= \left\langle \frac{D}{dt}\left(  \frac{\partial l}{\partial \bar{\xi}}\right) (\dot{x}(t)\oplus\bar{\xi}(t)), \bar{\eta}(t)\right\rangle+\left\langle \frac{\partial l}{\partial \bar{\xi}}(\dot{x}(t)\oplus\bar{\xi}(t)), \frac{D\bar{\eta}(t)}{dt}\right\rangle,
\end{align*}
and from the definition of associated affine connection
\begin{align*}
\frac{D\bar{\eta}(t)}{dt}&=[q(t),-[A(\dot{q}),\eta(t)]+\dot{\eta}]_G=[q(t),-[\xi(t),\eta(t)]]_G=-[\bar{\xi}(t),\bar{\eta}(t)]
\end{align*}
Finally, joining this last three equations, we get
\begin{align*}
\left\langle \frac{d}{dt}J(\dot{q}(t)),\eta\right\rangle &= \left\langle \frac{D}{dt}\left(  \frac{\partial l}{\partial \bar{\xi}}\right) (\dot{x}(t)\oplus\bar{\xi}(t)), \bar{\eta}(t)\right\rangle+\left\langle \frac{\partial l}{\partial \bar{\xi}}(\dot{x}(t)\oplus\bar{\xi}(t)), \frac{D\bar{\eta}(t)}{dt}\right\rangle\\
&=\left\langle \frac{D}{dt}\left(  \frac{\partial l}{\partial \bar{\xi}}\right) (\dot{x}(t)\oplus\bar{\xi}(t)), \bar{\eta}(t)\right\rangle-\left\langle \mathrm{ad}^*_{\bar{\xi}(t)}\frac{\partial l}{\partial \bar{\xi}}(\dot{x}(t)\oplus\bar{\xi}(t)), \bar{\eta}(t)\right\rangle \\
&=\mathrm{Ver}(\mathcal{LP})(l)\bar{\eta}(t).
\end{align*}
We have proved the following result.
\begin{proposition}
A curve $q(t)$ in $Q$ preserves the Noether current of a $G$-invarianf Lagrangian $L:TQ\to \mathbb{R}$ if and only if the curve $[\dot{q}]_G$ in $TQ/G\cong T(Q/G)\oplus\tilde{\mathfrak{g}}$ satisfies the vertical Lagrange--Poincar\'e equation.
\end{proposition}

%---------------------------------------------
%Pag 1, g_x no está introducido

\section{Reduction of variations in the $\mathfrak{LP}$ category}

\subsection{The Lagrange--Poincar\'e Category, $\mathfrak{LP}$} \label{Sect.LPcategory}
Iteration of the Lagrange--Poincar\'e reduction process has an immediate difficulty: the original Lagrangian $L$ is defined on $TQ$, the tangent bundle of the configuration space $Q$, while the reduced Lagrangian $l$ is defined on $TQ/G\cong T(Q/G)\oplus \tilde{\mathfrak{g}}$ which is not be a tangent bundle itself. Hence, reduction of $l$ involves the formulation of Lagrangian Mechanics in a wider category of bundles stable by the action of groups and consistent with the Lagrange-Poincar\'e reduction analyzed in the previous section.
%This was done in \cite{CMR}.

%DEFINICIÓN \label{LP}
The category of Lagrange--Poincar\'e bundles is denoted by $\mathfrak{LP}$ and is defined as follows:
\begin{enumerate}
\item \label{LP1} The \textbf{objects} of $\mathfrak{LP}$ are vector bundles $\tau_Q\oplus\tau:TQ	\oplus V\to Q$ where $\tau_Q:TQ\to Q$ is the tangent bundle of a manifold $Q$, and $\tau:V\to Q$ is a vector bundle with the following additional structure:
\begin{enumerate}
\item \label{LP1a} a Lie algebra on each fiber of $V$, denoted by $[,]$, such that $V$ is a Lie algebra bundle;
\item \label{LP1b} a $V$ valued 2-form $\omega$ on $Q$;
\item \label{LP1c}a covariant derivative $D/dt$ for curves in $V$ or equivalently a connection $\nabla$ on $V$;
\item \label{LP1d} the bilineal operator defined by
\begin{equation} \label{corchetesecciones}
[X_1\oplus w_1,X_2\oplus w_2]=[X_1,X_2]\oplus(\nabla_{X_1}w_2-\nabla_{X_2}w_1-\omega(X_1,X_2)+[w_1,w_2]),
\end{equation}
is a Lie bracket on sections $X\oplus w\in \Gamma(TQ\oplus V)$. Note that $[X_1,X_2]$ denotes the Lie bracket of vector fields while $[w_1,w_2]$ denotes the Lie bracket in the fibers of $V$.
\end{enumerate}
\item The \textbf{morphisms} between two Lagrange--Poincar\'e bundles $TQ_i\oplus V_i$, $i=1,2$ with structures $[,]_i$, $\omega_i$ and $D_i/dt$ are vector bundle morphisms $f:TQ_1\oplus V_1\to TQ_2\oplus V_2$ such that:
\begin{enumerate}
\item $f(TQ_1)\subset TQ_2$ and $f\vert_{TQ_1}=Tf_0$, where $f_0:Q_1\to Q_2$ is the function induced by $f$ in the base spaces;
\item \label{inv} $f(V_1)\subset V_2$ and $f\vert_{V_1}$ commutes with the additional structure, that is, given $v,v'\in (\tau_1)^{-1}(q)$, $X,X'\in (\tau_{Q_1})^{-1}(q)$ and a curve $v(t)$ in $V_1$:
$$f([v,v']_1)=[f(v),f(v')]_2,$$
$$f(\omega_1(X,X'))=\omega_2(f(X),f(X')),$$
$$f\left( \frac{D_1v(t)}{dt}\right) =\frac{D_2f(v(t))}{dt}.$$
\end{enumerate}
\end{enumerate}

Tangent bundles are a special case of Lagrange--Poincar\'e bundles for which $V=0$ is the trivial vector bundle. The Lie bracket on $\Gamma (TQ)$ is simply the Jacobi-Lie bracket for vector fields, and the morphisms between two tangent bundles are the tangent lift of functions. Another example of $\mathfrak{LP}$-bundle is $T(Q/G)\oplus \tilde{\mathfrak{g}}$ with the Lie bracket on $\tilde{\mathfrak{g}}$ defined in section \ref{req-bundles}, the curvature $\tilde{B}$ as the $2$-form on $Q/G$ and the covariant derivative induced by $\nabla^A$, the affine connection on $\tilde{\mathfrak{g}}$ obtained from the connection $A$ of $Q\to Q/G$  as an associated bundle. Furthermore, given two sections $X_i\oplus \bar{\xi}_i$ in $\Gamma(T(Q/G)\oplus \tilde{\mathfrak{g}})\equiv \mathfrak{X}^{\infty}(Q/G) \oplus \Gamma  (\tilde{\mathfrak{g}} )$ the expression
$$[X_1\oplus \bar{\xi}_1,X_2\oplus \bar{\xi}_2]=[X_1,X_2]\oplus(\nabla^A_{X_1}\bar{\xi}_2-\nabla^A_{X_2}\bar{\xi}_1-\tilde{B}(X_1,X_2)+[\bar{\xi}_1,\bar{\xi}_2]),$$
where the bracket on the first summand is the Jacobi-Lie bracket for tangent fields on $Q/G$, coincides with the quotient Lie bracket on $\Gamma ((TQ)/G)$.

One may think that condition 1.(d) can be deduced from the previous three condition and hence it is superflous. This is not the case. In fact, this condition can be rewritten in a less intriguing way imposing geometrical relations between $[,]$, $\omega$, and $\nabla$ as follows.
%Proposition
\begin{proposition} \label{3cond}
Let $\tau_Q\oplus\tau:TQ \oplus V\to Q$, where $\tau_Q:TQ\to Q$ is the tangent bundle of a manifold $Q$ and $\tau:V\to Q$ is a vector bundle, satisfying properties \em 1.(a)\em ,\em 1.(b)\em , and \em 1.(c)\em. Then, the expression
$$[X_1\oplus w_1,X_2\oplus w_2]=[X_1,X_2]\oplus(\nabla_{X_1}w_2-\nabla_{X_2}w_1-\omega(X_1,X_2)+[w_1,w_2])$$ defines a Lie bracket on sections $X\oplus w\in \Gamma(TQ\oplus V)$ if and only if
\begin{enumerate} [\em (a')\em ]
\setcounter{enumi}{3}
\item \label{LP1e} $d_{\nabla}\omega=0$;
\item \label{LP1f} $\nabla_X[v,w]=[\nabla_Xv,w]+[v,\nabla_X w]$ for all $X\in \mathfrak{X}(Q)$ and all $v,w\in \Gamma(V)$;
\item \label{LP1g} $k_{\nabla}(X,Y)v=-[\omega
(X,Y),v]$ for all $X,Y \in \mathfrak{X}(Q)$ and all $v\in \Gamma(V)$.
\end{enumerate}
\end{proposition}
\begin{proof}
The expression $[X_1,X_2]\oplus(\nabla_{X_1}w_2-\nabla_{X_2}w_1-\omega(X_1,X_2)+[w_1,w_2])$ is clearly $\mathbb{R}$-bilinear on $\Gamma(TQ\oplus V)$ and its skew symmetry is straightforward since
$$[X\oplus w,X\oplus w]=[X,X]\oplus(\nabla_{X}w-\nabla_{X}w-\omega(X,X)+[w,w])=0.$$
Thus, the expression defines a Lie Bracket if and only if it satisfies the Jacobi identity, that is, if the expression
\begin{align*}
& [ X_1\oplus w_1,[ X_2\oplus w_2,X_3\oplus w_3]]= \\=& \left[ X_1\oplus w_1,[X_2,X_3]\oplus(\nabla_{X_2}w_3-\nabla_{X_3}w_2-\omega(X_2,X_3)+[w_2,w_3])\right] \\
=&[X_1,[X_2,X_3]]\oplus (\nabla_{X_1}\nabla_{X_2} w_3-\nabla_{X_1}\nabla_{X_3} w_2-\nabla_{X_1}\omega(X_2,X_3)+\nabla_{X_1}[w_2,w_3]\\
-&\nabla_{[X_2,X_3]}w_1-\omega(X_1,[X_2,X_3])+[w_1, \nabla_{X_2} w_3-\nabla_{X_3} w_2-\omega(X_2,X_3)+[w_2,w_3]])
\end{align*}
is cyclic. This is equivalent to ask if
\begin{align*}
C(X_1,w_1,X_2,w_2,X_3,w_3)=&\nabla_{X_2}\nabla_{X_3} w_1-\nabla_{X_3}\nabla_{X_2} w_1-\nabla_{X_1}\omega(X_2,X_3)\\&+\nabla_{X_1}[w_2,w_3]-\nabla_{[X_2,X_3]}w_1-\omega(X_1,[X_2,X_3])\\&+[w_3, \nabla_{X_1} w_2]-[w_2,\nabla_{X_1} w_3]-[w_1,\omega(X_2,X_3)]
\end{align*}
is cyclic.
Assume that $C(X_1,w_1,X_2,w_2,X_3,w_3)$ is cyclic. For $w_1=w_2=w_3=0$ and arbitrary $X_1,X_2,X_3\in\mathfrak{X}(Q)$, we have $$C(X_i,0,X_j,0,X_k,0)=-\nabla_{X_i}\omega(X_j,X_k)-\omega(X_i,[X_j,X_k]),$$ for all $i,j,k\in\{1,2,3\}$. Hence,
\begin{align*}
0=C(X_1,0,X_2,0,X_3,0)+C(X_2,0,X_3,0,X_1,0)+C(X_3,0,X_1,0,X_2,0)=\\
-\nabla_{X_1}\omega(X_2,X_3)-\omega(X_1,[X_2,X_3])
-\nabla_{X_2}\omega(X_3,X_1)-\omega(X_2,[X_3,X_1])\\
-\nabla_{X_3}\omega(X_1,X_2)-\omega(X_3,[X_1,X_2])=-d_{\nabla}\omega(X_1,X_2,X_3),
\end{align*}
and we have condition 1.(d').

For $w_1=X_2=X_3=0$, we have
$$C(X_1,0,0,w_2,0,w_3)=[w_3, \nabla_{X_1} w_2]-[w_2,\nabla_{X_1} w_3]+\nabla_{X_1}[w_2,w_3],$$ and
$C(0,w_3,X_1,0,0,w_2)=C(0,w_2,0,w_3,X_1,0)=0$. Hence,
\begin{align*}
0&=C(X_1,0,0,w_2,0,w_3)+C(0,w_3,X_1,0,0,w_2)+C(0,w_2,0,w_3,X_1,0)\\
&=\nabla_{X_1}[w_2,w_3]-[\nabla_{X_1} w_2, w_3]-[w_2,\nabla_{X_1} w_3],
\end{align*}
from where 1.(e') follows.

For $X_1=w_2=w_3=0$, we have $$C(0,w_1,X_2,0,X_3,0)=\nabla_{X_2}\nabla_{X_3} w_1-\nabla_{X_3}\nabla_{X_2} w_1-\nabla_{[X_2,X_3]}w_1-[w_1,\omega(X_2,X_3)],$$ and
$C(X_3,0,0,w_1,X_2,0)= C(X_2,0,X_3,0,0,w_1)=0$. Hence,
\begin{align*}
0&=C(0,w_1,X_2,0,X_3,0)+C(X_3,0,0,w_1,X_2,0)+ C(X_2,0,X_3,0,0,w_1)\\
&=\nabla_{X_2}\nabla_{X_3} w_1-\nabla_{X_3}\nabla_{X_2} w_1-\nabla_{[X_2,X_3]}w_1-[w_1,\omega(X_2,X_3)]\\&=k_{\nabla}(X_2,X_3)w_1+[\omega(X_2,X_3),w_1]
\end{align*}
from where 1.(f') follows.

Conversely, conditions 1.(d'), 1.(e') and 1.(f') imply that expressions $C(X_1,0,X_2,0,X_3,0)$, $C(X_1,0,0,w_2,0,w_3)$, and $C(0,w_1,X_2,0,X_3,0)$ are cyclic. Hence,
\begin{eqnarray*}
C(X_1,w_1,X_2,w_2,X_3,w_3)=&C(X_1,0,X_2,0,X_3,0)+C(X_1,0,0,w_2,0,w_3)\\
&+C(0,w_1,X_2,0,X_3,0)
\end{eqnarray*} is cyclic.
\end{proof}

An action in the category $\mathfrak{LP}$ of a Lie group $G$ on an object $TQ\oplus V$ of $\mathfrak{LP}$ is a vector bundle action $\rho:G\times TQ\oplus V \to TQ\oplus V$ such that for all $g\in G$, $\rho_g:TQ\oplus V \to TQ\oplus V$ is a $\mathfrak{LP}$ isomorphism. As a consequence, the additional structures  $[,]$, $\omega$, $\nabla$ are invariant by $\rho$ and quotient structures can be defined in the sense specified below.

A Lie bracket $[,]$ on $V$ is said to be invariant if for all $g\in G$ and all $v_1,v_2\in V$ such that $\tau(v_1)= \tau(v_2)$ it follows $g[v_1,v_2]=[gv_1,gv_2].$ Then, the expression
$$[[v_1]_G,[v_2]_G]_G=[[v_1,v_2]]_G$$ defines a \textit{quotient Lie bracket} on $V/G$.

A $V$-valued $2$-form $\omega$ on $Q$ is said to be invariant if for all $g\in G$ and all $X,Y\in TQ$ such that $\tau_Q(X)=\tau_Q(Y)$ we have that $g\omega(X,Y)=\omega(gX,gY)$.
An invariant form defines a \textit{quotient generalized  form} as
$$[\omega]_G([X]_G,[Y]_G)=[\omega(X,Y)]_G$$
for all $X,Y\in TQ$ such that $\tau_Q(X)=\tau_Q(Y).$ Observe that $[\omega]_G$ is not a form on $Q/G$ since it is a skew-symmetric bilineal form on the fibers of $TQ/G$ rather than on the fibers of $T(Q/G)$. In fact, we identify $[X]_G\equiv T\pi (X) \oplus \bar{\xi}$, $[Y]_G\equiv T\pi (Y) \oplus \bar{\eta}$ via the isomorphism $\alpha_A$ and obtain
\begin{equation*}
\begin{split}
[\omega]_G([X]_G,[Y]_G)=&[\omega]_G(T\pi (X),T\pi (Y))+[\omega]_G(T\pi (X),\bar{\eta})\\&+[\omega]_G(\bar{\xi},T\pi (Y))+[\omega]_G(\bar{\xi},\bar{\eta}).
\end{split}
\end{equation*}
Only the term $[\omega]_G(T\pi (X),T\pi (Y))$ is a $2$-form on $Q/G$.

Let $Z=X\oplus \bar{\xi}\in \Gamma (TQ/G)=\Gamma (T(Q/G))\oplus\Gamma(\tilde{\mathfrak{g}})$ and $[v]_G\in \Gamma(V/G)$ with $v\in\Gamma^G(V)$. There is a unique $\bar{Z}\in\Gamma^G(TQ)$ identified with $Z$. Furthermore, $\bar{Z}=X^h\oplus Y$ with $X^h \in\mathfrak{X}(TQ)$ the horizontal lift of $X$ and $Y$ the unique $\pi$-vertical $G$-invariant vector field such that for all $x\in Q/G$, $\bar{\xi}(x)=[q,A(Y(q))]_G$ with $q\in \pi^{-1}(x)$. Then the \textit{quotient connection} is defined by
$$\left[ \nabla^{(A)}\right] _{G,X\oplus \bar{\xi}}[v]_G=[\nabla_{\bar{Z}}v]_G,$$
the \textit{vertical quotient connection} is defined by
$$\left[ \nabla^{(A,V)}\right] _{G,X\oplus \bar{\xi}}[v]_G=[\nabla_{Y}v]_G,$$
and the \textit{horizontal quotient connection} is defined by
$$\left[ \nabla^{(A,H)}\right] _{G,X\oplus \bar{\xi}}[v]_G=[\nabla_{X^h}v]_G.$$
Note that quotient connections are not connections in the usual sense since derivation is carried along a section of $TQ/G$ intead of a section of $T(Q/G)$. In particular, this is called a $(TQ)/G$-connection in the context of Lie algebroids \cite{algebroids2017}. Yet, the horizontal quotient connection can be thought as a usual connection on $T(Q/G)$ since it only depends on $T(Q/G)\subset TQ/G$.

Using this quotient structures, it is proved in \cite{CMR} that the $\mathfrak{LP}$ category is stable by reduction:
\begin{theorem} \label{quotientLP}
Let $\tau_Q\oplus \tau:TQ\oplus V \to Q$ be an object of $\mathfrak{LP}$ with additional structures $[,]$, $\omega$, $\nabla$. Let $\rho:G\times (TQ\oplus V)\to TQ\oplus V$ be an action in the category $\mathfrak{LP}$ and $A$ a principal connection on $Q\to Q/G$. Then, the vector bundle
$$T(Q/G)\oplus\tilde{\mathfrak{g}}\oplus(V/G)$$
with additional structures $[,]^{\tilde{\mathfrak{g}}}$, $\omega^{\tilde{\mathfrak{g}}}$ and $\nabla^{\tilde{\mathfrak{g}}}$ in $\tilde{\mathfrak{g}}\oplus(V/G)$ given by
\begin{eqnarray*}
\nabla^{\tilde{\mathfrak{g}}}_X(\bar{\xi}\oplus[v]_G)&=&\nabla^{A}_X\bar{\xi}\oplus\left([\nabla^{(A,H)}]_{G,X}[v]_G-[\omega]_G(X,\bar{\xi}) \right),\\
\omega^{\tilde{\mathfrak{g}}}(X_1,X_2)&=&\tilde{B}^A(X_1,X_2)\oplus [\omega]_G(X_1,X_2), \\
\,[\bar{\xi}_1\oplus[v_1]_G,\bar{\xi}_2\oplus[v_2]_G]^{\tilde{\mathfrak{g}}} &=& [\bar{\xi}_1,\bar{\xi}_2]\\
\oplus \left([\nabla^{(A,V)}]_{G,\bar{\xi}_1}[v_2]_G \right. & -  &\left. [\nabla^{(A,V)}]_{G,\bar{\xi}_2}[v_1]_G
-[\omega]_G(\bar{\xi}_1,\bar{\xi}_2)+[[v_1]_G,[v_2]_G]_G  \right)
\end{eqnarray*}
is an object of the $\mathfrak{LP}$ category, the reduced bundle with respect to the group $G$ and the connection $A$.
\end{theorem}

\subsection{Lagrangian mechanics in the $\mathfrak{LP}$ category}
Let $TQ\oplus V\in \mathfrak{LP}$ and let
$$L:TQ\oplus V\to \mathbb{R}$$
be a function, the Lagrangian.
We denote by $\ell\Omega(Q)\oplus\Omega(V)$ the space of curves in $TQ\oplus V$ of the form $\dot{q}(t)\oplus v(t)$ where $v(t)$ is a curve in $V$, $q(t)=\tau(v(t))$ is a curve in $Q$ and $\dot{q}(t)$ the lift of $q(t)$ to $TQ$. In addition, $\ell\Omega(Q;q_0,q_1)\oplus\Omega(V;q_0,q_1)$ denotes the curves within $\Omega(Q)\oplus\Omega(V)$ such that $q(t)$ has endpoints $q_0,q_1$. The action of $L:TQ\oplus V\to \mathbb{R}$ is defined on $\ell\Omega(Q;q_0,q_1)\oplus\Omega(V;q_0,q_1)$ as
\begin{equation*}
S(L)(\dot{q}\oplus v)=\int_{t_0}^{t_1} L(\dot{q}(t),v(t)) dt.
\end{equation*}
Let $\Delta^{\ell}_{q\oplus v}\subset \Delta_{\dot{q}\oplus v}$ be the set of variations of the form $\delta \dot{q}\oplus  \delta v$ with $\delta \dot{q}\in \Delta^{\ell}_{q}$ and
$$\delta v=\frac{Dw}{dt}+[v,w]+\omega_q(\delta q,\dot{q}),$$
where $w(t)$ is a curve in $V$ with zero endpoints and $\tau(w(t))=q(t)$.
A curve $\dot{q}(t)\oplus v(t)\in \ell\Omega(Q;q_0,q_1)\oplus\Omega(V;q_0,q_1)$ satisfies the variational principle in $\mathfrak{LP}$ defined by $L: TQ\oplus V\to \mathbb{R}$ if for any $\delta (\dot{q}\oplus v)\in \Delta^{\ell}_{q\oplus v}$ we have that
$$0=d\mathcal{S}(L)\cdot (\delta \dot{q}\oplus  \delta v)=\left.\frac{d}{d\lambda}\right|_{\lambda =0}\mathcal{S}(L)((\dot{q}\oplus v)(t,\lambda)),$$
where $(\dot{q}\oplus v)(t,\lambda)$ is a deformation of $\dot{q}\oplus v$ inducing $\delta \dot{q}\oplus  \delta v$.

This variational principle for Lagrangians defined on $\mathfrak{LP}$ bundles can be translated into equations.

%PROPOSICIÓN
\begin{theorem} \label{ecs.lp}
Let $L:TQ\oplus V\to\mathbb{R}$ a Lagrangian and
$$S(L)(\dot{q}\oplus v)=\int_{t_0}^{t_1} L(\dot{q}(t),v(t)) dt$$
its action on curves $\dot{q}(t)\oplus v(t)$ within $\ell\Omega(Q;q_0,q_1)\oplus\Omega(V;q_0,q_1)$. There exists a unique bundle map
$$\mathcal{LP}(L):T^{(2)}Q\oplus 2 V\to T^*Q\oplus V^*$$
such that, for each variation $\delta \dot{q}\oplus \delta v$ in $\Delta^{\ell}_{q\oplus v}$, it satisfies
$$d\mathcal{S}(L)\cdot (\delta \dot{q}\oplus \delta v)= \int_{t_0}^{t_1} \mathcal{LP}(L)([q]^{(2)}\oplus [v]^{(1)})\cdot (\delta q\oplus \delta v).$$
The bundle map $\mathcal{LP}(L)$ is called Lagrange--Poincar\'e operator and since it takes values in a direct sum, it can be decomposed into two terms: the horizontal Lagrange--Poincar\'e operator $\mathrm{Hor}(\mathcal{LP})(L)$ and the vertical  Lagrange--Poincar\'e operator $\mathrm{Ver}(\mathcal{LP})(L)$.
Their expressions are
\begin{align*}
\mathrm{Ver}\mathcal{LP}(L):T^{(2)}Q\oplus 2V &\to T^*Q \\
[q]^{(2)}\oplus [v]^{(1)} &\mapsto \frac{\partial L}{\partial q}-\frac{D}{dt}\frac{\partial L}{\partial \dot{q}}-\left\langle\frac{\partial L}{\partial v},\omega_q(\dot{q},\cdot)\right\rangle,
\end{align*}
\begin{align*}
\mathrm{Hor}\mathcal{LP}(L):T^{(2)}Q\oplus 2V &\to V^* \\
[q]^{(2)}\oplus [v]^{(1)} &\mapsto  \mathrm{ad}^*_{v}\frac{\partial L}{\partial v}-\frac{D}{dt}\frac{\partial L}{\partial v} .
\end{align*}

\end{theorem}

\begin{proof}
The proof is analogous to the explicit deduction of the usual Lagrange--Poincar\'e equations. We apply the chain rule to the factorization,
\begin{align*}
\mathbb{R} &\longrightarrow& TQ \oplus V &\longrightarrow& \mathbb{R} \\
\lambda &\longmapsto& (\dot{q}\oplus v)(t,\lambda) &\longmapsto& L((\dot{q}\oplus v)(t,\lambda)) \\
\end{align*}
to obtain $$\frac{d L(\dot{q}\oplus v)(t,\lambda)}{d\lambda}=\frac{\partial L}{\partial (\dot{q}\oplus v)}\frac{\partial(\dot{q}\oplus v)(t,\lambda)}{\partial \lambda},$$
where
 \begin{eqnarray*}
\partial\dot{q}\oplus v(t,\lambda)/\partial \lambda &:&\mathbb{R}\to T_{\dot{q}\oplus v(t,\lambda)}(TQ\oplus V),\\
\partial L/\partial (\dot{q}\oplus v) &:& T_{\dot{q}\oplus v(t,\lambda)}(TQ\oplus V)\to \mathbb{R}.
\end{eqnarray*}
We chose an arbitrary connection on $TQ\to Q$ which, together with the connection $\nabla$ on $V$, defines a connection on $TQ\oplus V$. This connection provides a split of $T_{\dot{q}\oplus v(t,\lambda)}(TQ\oplus V)$ into
\begin{align*}
&\mathrm{Hor}_{\dot{q}\oplus v(t,\lambda)}(TQ\oplus V)\oplus \mathrm{Ver^1}_{\dot{q}\oplus v(t,\lambda)}(TQ\oplus V)\oplus \mathrm{Ver^2}_{\dot{q}\oplus v(t,\lambda)}(TQ\oplus V)\\&\cong T_{q(t,\lambda)}Q\oplus T_{q(t,\lambda)}Q\oplus V_{q(t,\lambda)}
\end{align*}
Then,

\begin{align*}
&d\mathcal{S}(L)\cdot (\delta \dot{q}\oplus \delta v)= \left.\frac{d}{d\lambda}\right|_{\lambda =0}\mathcal{S}(L)((\dot{q}\oplus v)(t,\lambda))\\=&\left.\frac{d}{d\lambda}\right|_{\lambda =0} \left( \displaystyle\int_{t_0}^{t_1} L((\dot{q}\oplus v)(t,\lambda))dt\right)
= \displaystyle\int_{t_0}^{t_1}\left.\frac{d}{d\lambda}\right|_{\lambda =0}L((\dot{q}\oplus v)(t,\lambda))dt\\
=& \displaystyle\int_{t_0}^{t_1} \left( \frac{\partial L}{\partial q} \left.\frac{\partial q}{\partial \lambda}\right|_{\lambda =0}+\frac{\partial L}{\partial \dot{q}}\left.\frac{D\dot{q}}{d\lambda}\right|_{\lambda =0}+\frac{\partial L}{\partial v}\left.\frac{Dv}{d\lambda}\right|_{\lambda =0}\right) dt.\\
\end{align*}
Performing integration by parts as in the usual Euler-Lagrang\'e proof gives

\begin{align*}
d\mathcal{S}(L)\cdot (\delta \dot{q}\oplus \delta v)
=& \displaystyle\int_{t_0}^{t_1} \left( \left( \frac{\partial L}{\partial q}-\frac{D}{dt}\left( \frac{\partial L}{\partial \dot{q}}\right)   \right) \left.\frac{\partial q}{\partial \lambda}\right|_{\lambda =0}+\frac{\partial L}{\partial v}\left.\frac{Dv}{d\lambda}\right|_{\lambda =0}\right)dt.\\
\end{align*}
As $\delta \dot{q}\oplus \delta v\in \Delta^{\ell}_{q\oplus v}$, we have that
$$\frac{Dv}{d\lambda}=\frac{Dw}{dt}+[v,w]+\omega_q(\delta q,\dot{q}),$$
where $w(t)$ is a curve in $V$ with null endpoints  and $\tau (w(t))=q(t)$. Hence;
\begin{align*}
&d\mathcal{S}(L)\cdot (\delta \dot{q}\oplus \delta v)
= \displaystyle\int_{t_0}^{t_1} \left( \left( \frac{\partial L}{\partial q}-\frac{D}{dt}\left( \frac{\partial L}{\partial \dot{q}}\right)   \right) \delta q+\frac{\partial L}{\partial v}\left( \frac{Dw}{dt}+[v,w]+\omega_q(\delta q,\dot{q})\right) \right)dt\\
=& \displaystyle\int_{t_0}^{t_1} \left( \left(\frac{\partial L}{\partial q}-\frac{D}{dt}\left( \frac{\partial L}{\partial \dot{q}}\right)   \right) \delta q
+\frac{d}{dt}\left( \frac{\partial L}{\partial v}w\right)-\frac{D}{dt}\left(\frac{\partial L}{\partial v} \right)w
+\frac{\partial L}{\partial v}
[v,w]+\frac{\partial L}{\partial v}\omega_q(\delta q,\dot{q}) \right)dt\\
=&\left[ \frac{\partial L}{\partial v}w\right]^{t_1}_{t_0}+ \displaystyle\int_{t_0}^{t_1} \left( \left( \frac{\partial L}{\partial q}-\frac{D}{dt}\left( \frac{\partial L}{\partial \dot{q}}\right)   \right) \delta q
-\frac{D}{dt}\left(\frac{\partial L}{\partial v} \right)w+\mathrm{ad}^*_{v}\frac{\partial L}{\partial v}
w-\frac{\partial L}{\partial v}\omega_q(\dot{q},\delta q) \right) dt\\
=& \displaystyle\int_{t_0}^{t_1} \left( \left( \frac{\partial L}{\partial q}-\frac{D}{dt}\left( \frac{\partial L}{\partial \dot{q}}\right)   \right) \delta q -\frac{\partial L}{\partial v}\omega_q(\dot{q},\delta q)
+\left( \mathrm{ad}^*_{v}\frac{\partial L}{\partial v}-\frac{D}{dt}\left(\frac{\partial L}{\partial v} \right)\right)
w \right) dt\\
=& \displaystyle\int_{t_0}^{t_1}\mathcal{LP}(L)([q]^{(2)}\oplus [v]^{(1)})\cdot (\delta q\oplus \delta v),
\end{align*}
and the proof is complete.
\end{proof}

Therefore a curve $\dot{q}(t)\oplus v(t)$ of $\ell\Omega(Q;q_0,q_1)\oplus\Omega(V,q_0,q_1)$ satisfies the variational principle if and only if it satisfies the Lagrange--Poincar\'e equations
\begin{eqnarray}
\frac{\partial L}{\partial q}-\frac{D}{dt}\frac{\partial L}{\partial \dot{q}}-\left\langle\frac{\partial L}{\partial v},\omega_q(\dot{q},\cdot)\right\rangle &=&0,\label{EPV}\\
\mathrm{ad}^*_{v}\frac{\partial L}{\partial v}-\frac{D}{dt}\left(\frac{\partial L}{\partial v}\right)&=&0\label{EPV}.
\end{eqnarray}
%$$0=\mathcal{LP}(L)(\delta q\oplus \delta v).$$
%Or equivalently, $\mathrm{Ver}(\mathcal{LP})(L) \delta v=0$ y $\mathrm{Hor}(\mathcal{LP})(L)\delta q=0.$

\subsection{Reduction of Variations}

Given an action of a Lie group $G$ on an element $TQ\oplus V$ in the category $\mathfrak{LP}$, there is a reduction process analogous to the reduction of Lagrangians defined on tangent bundles. This process is established for $\mathfrak{RI}$, the smallest subcategory of $\mathfrak{LP}$ that contains tangent bundles and is closed under the quotiening operation in \cite{CMR}. The extension of this procedure to the whole $\mathfrak{LP}$ category given below.

We choose a connection $A$ in the principal bundle $Q\to Q/G$ and we consider the identification
\begin{align*}
\alpha^{TQ\oplus V}_A: (TQ\oplus V )/G &\to T(Q/G)\oplus \tilde{\mathfrak{g}}\oplus (V/G)\\
[\dot{q}\oplus v]_G&=\alpha_A(\dot{q})\oplus [v]_G.
\end{align*}
Furthermore, $\alpha^{TQ\oplus V}_A$ applied to sections
is a Lie algebra homomorphism between $\Gamma((TQ\oplus V )/G)$ equipped with the quotient Lie bracket of $\Gamma(TQ\oplus V)$ and $\Gamma (T(Q/G)\oplus \tilde{\mathfrak{g}}\oplus (V/G))$ with the Lie bracket induced by the additional structures of Theorem \ref{quotientLP}.

For a Lagrangian $L:TQ\oplus V \to \mathbb{R}$ invariant by the action of $G$, with the identification above, we define a reduced Lagrangian
$$L^{(G)}:T(Q/G)\oplus \tilde{\mathfrak{g}}\oplus (V/G) \to \mathbb{R}.$$

First, we recall a lemma of reconstruction of curves from \cite{CMR}:
\begin{lemma} \label{curves}
Let $\pi_G$ be the projection of $TQ\oplus V$ to $(TQ\oplus V)/G$, the map
\begin{equation*}
\begin{split}
\Omega(\alpha^{TQ\oplus V}_A\circ \pi_G):\Omega (TQ\oplus V)\to &\Omega(T(Q/G)\oplus\tilde{\mathfrak{g}}\oplus V/G)\\
\gamma(t) \mapsto & (\alpha^{TQ\oplus V}_A\circ \pi_G) (\gamma(t))
\end{split}
\end{equation*}
restricted to $\ell\Omega(Q;q_0)\oplus \Omega (V;q_0)$ is  inyective and the image of its restriction is $$\ell\Omega(Q/G;x_0)\oplus \Omega(\tilde{\mathfrak{g}};x_0)\oplus\Omega(V/G;x_0).$$
\end{lemma}
%\begin{proof}
%Dada una curva $\dot{x}(t)\oplus\bar{\xi}(t)\oplus [v]_G(t)$ con $x(t_0)=x_0$, por el lema \ref{recon1}  existe una única curva $q\in \Omega(Q;q_0)$ tal que $\Omega(\alpha_A\circ \pi_G)(q)=\dot{x}\oplus\bar{\xi}$. Por otro lado, since las fibras de $V$ son isomorfas a las de $V/G$ de modo que existe una única curva $v\in\Omega(V)$ tal que $\tau v(t)=q(t)$ y su clase es $[v]_G(t)$ para todo $t$. Así, la curva buscada es $q\oplus v$.
%\end{proof}
As a corollary of this lemma, $\Omega(\alpha^{TQ\oplus V}_A\circ \pi_G)$ is a bijection between $\ell\Omega(Q;q_0,q_1)\oplus \Omega (V;q_0,q_1)$ and $\ell\Omega(Q/G;x_0,x_1)\oplus \Omega(\tilde{\mathfrak{g}};x_0,x_1)\oplus\Omega(V/G;x_0,x_1),$ that is, the sets of curves considered in both variational problems. A similar result for allowed variations requires a careful study of the geometry of the reduced variations in $\tilde{\mathfrak{g}}\oplus V/G$ as follows.
\begin{theorem} \label{variations}
Let $\dot{q}\oplus v$ be a curve in $\ell\Omega(Q;q_0,q_1)\oplus \Omega (V;q_0,q_1)$ and $\dot{x}\oplus \bar{\xi}\oplus [v]_G=\alpha^{TQ\oplus V}_A\circ \pi_G(\dot{q}\oplus v)$. Then the map
\begin{align*}
T_{\dot{q}\oplus v}\Omega(\alpha^{TQ\oplus V}_A\circ \pi_G):\Delta^{\ell}_{q\oplus v}&\to \Delta_{x\oplus \bar{\xi}\oplus [v]_G}\\
\delta\dot{q}\oplus\delta v &\mapsto \left.\frac{d}{d\lambda}\right|_{\lambda =0} (\alpha^{TQ\oplus V}_A\circ \pi_G)(\dot{q}(t,\lambda)\oplus v(t,\lambda)),
\end{align*}
where $\dot{q}(t,\lambda)\oplus v(t,\lambda)$ is any deformation of $\dot{q}\oplus v$ inducing the variation $\delta\dot{q}\oplus\delta v$, is a linear isomorphism from $\Delta^{\ell}_{q\oplus v}$ onto $\Delta^{\ell}_{x\oplus \bar{\xi}\oplus [v]_G}$.
%=\frac{d}{d\lambda}\big\vert_{\lambda=0} \dot{q}(t,\lambda)\oplus v(t,\lambda)
\end{theorem}
\begin{proof} We separate this argument in different steps.

(1). First, we write explicitly the variations $\Delta^{\ell}_{x\oplus \bar{\xi}\oplus [v]_G}$, that is, the allowed variations of a curve $\dot{x}\oplus \bar{\xi}\oplus [v]_G$ in $T(Q/G)\oplus(\tilde{\mathfrak{g}}\oplus V/G)$ with additional structures $[,]^{\tilde{\mathfrak{g}}}$, $\omega^{\tilde{\mathfrak{g}}}$ and $\nabla^{\tilde{\mathfrak{g}}}$. Whereas $\delta x$ is free, the variation $\delta\bar{\xi}\oplus \delta [v]_G$ is given by
$$\delta\bar{\xi}\oplus \delta [v]_G=\frac{D^{\tilde{\mathfrak{g}}}}{dt}(\bar{\eta}\oplus[w]_G)+[\bar{\xi}\oplus [v]_G,\bar{\eta}\oplus[w]_G]^{\tilde{\mathfrak{g}}}-\omega^{\tilde{\mathfrak{g}}}(\dot{x},\delta x),$$
where $\bar{\eta}\oplus[w]_G$ is a curve in $\tilde{\mathfrak{g}}\oplus V/G$ with null endpoints such that $\tau_G(\bar{\eta}(t)\oplus[w]_G(t))=x(t)$ . From the explicit expression of the additional structures in Theorem \ref{quotientLP},
$$\frac{D^{\tilde{\mathfrak{g}}}}{dt}(\bar{\eta}\oplus[w]_G)=\nabla^{\tilde{\mathfrak{g}}}_{\dot{x}}(\bar{\eta}\oplus[w]_G)=\nabla_{\dot{x}}\bar{\eta}\oplus ([\nabla^{(A,H)}]_{G,\dot{x}}[w]_G-[\omega]_G(\dot{x},\bar{\eta}));$$
\begin{equation*}
\begin{split}
[\bar{\xi}\oplus [v]_G,\bar{\eta}\oplus[w]_G]^{\tilde{\mathfrak{g}}}=[\bar{\xi},\bar{\eta}]\oplus ([\nabla^{(A,V)}]_{G,\bar{\xi}}&[w]_G-[\nabla^{(A,V)}]_{G,\bar{\eta}}[v]_G\\ &-[\omega]_G(\bar{\xi},\bar{\eta})+[[v]_G,[w]_G]_G);
\end{split}
\end{equation*}

$$\omega^{\tilde{\mathfrak{g}}}(\dot{x},\delta x)=\tilde{B}^A(\dot{x},\delta x)\oplus [\omega]_G(\dot{x},\delta x)$$
Hence, the variation of $\bar{\xi}$ is given by
$$\delta\bar{\xi}=\frac{D \bar{\eta}}{dt}+[\bar{\xi},\bar{\eta}]-\tilde{B}^A(\dot{x},\delta x)$$
and the variation of $[v]_G$ is given by
\begin{align*}
\delta [v]_G=[\nabla^{(A,H)}]_{G,\dot{x}}[w]_G-[\omega]_G&(\dot{x},\bar{\eta})+[\nabla^{(A,V)}]_{G,\bar{\xi}}[w]_G-[\nabla^{(A,V)}]_{G,\bar{\eta}}[v]_G\\&-[\omega]_G(\bar{\xi},\bar{\eta})+[[v]_G,[w]_G]_G-[\omega]_G(\dot{x},\delta x).
\end{align*}

(2). We now prove that
$$T_{\dot{q}\oplus v}\Omega(\alpha^{TQ\oplus V}_A\circ \pi_G)(\Delta^{\ell}_{q\oplus v})\subset \Delta^{\ell}_{x\oplus \bar{\xi}\oplus [v]_G}.$$
A variation in $\Delta^{\ell}_{q\oplus v}$ is obtained as the derivative at $\lambda=0$ of a deformation  $\dot{q}(t,\lambda)\oplus v(t,\lambda)$ of $\dot{q}(t)\oplus v(t)$ such that for each $\lambda$, $\dot{q}_{\lambda}(t)\oplus v
_{\lambda}(t)$ is a curve belonging to $\ell\Omega(Q;q_0,q_1)\oplus \Omega (V;q_0,q_1)$. We study the variation in $T(Q/G)\oplus\tilde{\mathfrak{g}}\oplus V/G$ obtained from the deformation
$$\dot{x}(t,\lambda)\oplus \bar{\xi}(t,\lambda)\oplus [v(t,\lambda)]_G=\alpha^{TQ\oplus V}_A\circ \pi_G(\dot{q}(t,\lambda)\oplus v(t,\lambda)).$$
As seen in subsection \ref{usuallp}, the variation $\delta x$ is free and $\delta \dot{x}\in \Delta^{\ell}_x$. Furthermore, observe that the horizontal part of $$\left.\frac{d}{d\lambda}\right|_{\lambda =0}\left( \bar{\xi}(t,\lambda)\oplus [v(t,\lambda)]_G\right) $$ coincides with the horizontal lift of $\delta x$. Consequently, we calculate directly the covariant derivative with respect to the connection $\nabla^{\tilde{\mathfrak{g}}}$ of $\tilde{\mathfrak{g}}\oplus V/G$. In turn, this derivative can be expressed in terms of the connection as

$$\left.\frac{D^{\tilde{\mathfrak{g}}}}{d\lambda}\right|_{\lambda =0}\left( \bar{\xi}(t,\lambda)\oplus [v(t,\lambda)]_G\right)=\nabla^{\tilde{\mathfrak{g}}}_{\delta x}(\bar{\xi}\oplus [v]_G). $$
Since
\begin{align*}
\nabla^{\tilde{\mathfrak{g}}}_{\delta x}(\bar{\xi}\oplus [v]_G)&=\nabla^{A}_{\delta x}\bar{\xi}\oplus ([\nabla^{(A,H)}]_{G,\delta x}[v]_G-[\omega]_G(\delta x,\bar{\xi}))\\
&=\nabla^{A}_{\delta x}\bar{\xi}\oplus ([\nabla^{(A)}]_{G,\delta x\oplus\bar{\eta}}[v]_G-[\nabla^{(A,V)}]_{G,\bar{\eta}}[v]_G-[\omega]_G(\delta x,\bar{\xi}))\\
&=\nabla^{A}_{\delta x}\bar{\xi}\oplus ([\nabla^{(A)}]_{G,\delta q}[v]_G-[\nabla^{(A,V)}]_{G,\bar{\eta}}[v]_G-[\omega]_G(\delta x,\bar{\xi})),
\end{align*}
where $\bar{\eta}$ is a curve in  $\tilde{\mathfrak{g}}$ such that $\delta q=\delta x\oplus \bar{\eta}$, we conclude that for $\lambda=0$
$$\frac{D^{\tilde{\mathfrak{g}}}}{d\lambda}\left( \bar{\xi}(t,\lambda)\oplus [v(t,\lambda)]_G\right)=\frac{D\bar{\xi}}{d\lambda}\oplus \left( \left[ \frac{D}{d\lambda}v(t,\lambda)\right]_G-[\nabla^{(A,V)}]_{G,\bar{\eta}}[v]_G-[\omega]_G(\delta x,\bar{\xi})\right).$$
Similarly as in section \ref{usuallp}
$$\left.\frac{D}{d\lambda}\right|_{\lambda =0}\bar{\xi}=\frac{D \bar{\eta}}{dt}+[\bar{\xi},\bar{\eta}]-\tilde{B}^A(\dot{x},\delta x).$$ As $\delta q\oplus \delta v$ is an allowed variation of $TQ\oplus V$, $$\delta v=\left.\frac{D}{d\lambda}\right|_{\lambda =0}v=\frac{Dw}{dt}+[v,w]+\omega_q(\delta q,\dot{q}),$$
where $w(t)$ is a curve in $V$ such that $w(t_0)=w(t_1)=0$ and $\tau\circ w=q$. We study now the projection to  $V/G$ of each of the terms of $\delta v$:
\begin{align*}\left[ \frac{Dw}{dt}\right]_G=&[\nabla^A_{\dot{q}}w]_G=[\nabla^A_{\dot{x}\oplus\bar{\xi}}w]_G=[\nabla^{(A)}]_{G,\dot{x}\oplus\bar{\xi}}[w]_G\\=&[\nabla^{(A,H)}]_{G,\dot{x}}[w]_G+[\nabla^{(A,V)}]_{G,\bar{\xi}}[w]_G;
\end{align*}
$$[[v,w]]_G=[[v]_G,[w]_G]_G;$$
\begin{align*}
[\omega(\dot{q},\delta q)]_G&=[\omega]_G([\dot{q}]_G,[\delta q]_G)=[\omega]_G(\dot{x}\oplus \bar{\xi},\delta x\oplus \bar{\eta})\\&=[\omega]_G(\dot{x},\delta x)+[\omega]_G(\bar{\xi},\delta x)
+[\omega]_G(\dot{x},\bar{\eta})+[\omega]_G(\bar{\xi}, \bar{\eta}).
\end{align*}
Consequently,
\begin{align*}\label{projvarv}
\left[\frac{Dv}{d\lambda}\big\vert_{\lambda=0}\right]_G=&[\nabla^{(A,H)}]_{G,\dot{x}}[w]_G+[\nabla^{(A,V)}]_{G,\bar{\xi}}[w]_G+[[v]_G,[w]_G]_G\\&-[\omega]_G(\dot{x},\delta x)+[\omega]_G(\delta x,\bar{\xi})
-[\omega]_G(\dot{x},\bar{\eta})-[\omega]_G(\bar{\xi}, \bar{\eta}).
\end{align*}
Finally, we substitute this expression into the covariant derivative of $\bar{\xi}(t,\lambda)\oplus [v(t,\lambda)]_G$ and obtain the variation $\delta\bar{\xi}\oplus \delta [v]_G$.

\begin{equation*}
\begin{aligned}
\delta\bar{\xi}\oplus \delta [v]_G=\frac{D^{\tilde{\mathfrak{g}}}}{d\lambda}\left( \bar{\xi}(t,\lambda)\oplus [v(t,\lambda)]_G\right)&=\left( \frac{D \bar{\eta}}{dt}+[\bar{\xi},\bar{\eta}]-\tilde{B}^A(\dot{x},\delta x)\right) \oplus \\ \big( [\nabla^{(A,H)}]_{G,\dot{x}}[w]_G+[\nabla^{(A,V)}]_{G,\bar{\xi}}[w]_G&-[\nabla^{(A,V)}]_{G,\bar{\eta}}[v]_G +[[v]_G,[w]_G]_G \\&-[\omega]_G(\dot{x},\delta x)
-[\omega]_G(\dot{x},\bar{\eta})-[\omega]_G(\bar{\xi}, \bar{\eta}) \big).
\end{aligned}
\end{equation*}
The variation obtained via $\alpha^{TQ\oplus V}_A\circ \pi_G$ coincides with the variation of $\bar{\xi}(t,\lambda)\oplus [v(t,\lambda)]_G$ obtained from $\bar{\eta}\oplus[w]_G$ in $T(Q/G)\oplus\tilde{\mathfrak{g}}\oplus V/G$.

(3). The map $T_{\dot{q}\oplus v}\Omega(\alpha^{TQ\oplus V}_A\circ \pi_G)$ does not depend on the chosen deformation and is clearly linear. Then, it only remains to prove its bijectivity. Let
$$\dot{x}(t,\lambda)\oplus \bar{\xi} (t,\lambda)\oplus [v]_G(t,\lambda)$$
be an allowed deformation of $\dot{x}(t)\oplus \bar{\xi} (t)\oplus [v]_G(t)$. From Lemma \ref{curves}, for each  $\lambda$ there is a unique curve $\dot{q}_{\lambda}(t)\oplus v_{\lambda}(t)$ in $\ell\Omega(Q;q_0,q_1)\oplus\Omega(V,q_0,q_1)$ such that its image by $\Omega(\alpha_A\circ \pi_G)$ is $\dot{x}(t,\lambda)\oplus \bar{\xi} (t,\lambda)\oplus [v]_G(t,\lambda)$. Since $\dot{x}(t,0)\oplus \bar{\xi} (t,0)\oplus [v]_G(t,0)=\dot{x}(t)\oplus \bar{\xi} (t)\oplus [v]_G(t)$, uniqueness implies that $\dot{q}_{0}(t)\oplus v_{0}(t)=\dot{q}(t)\oplus v(t)$ and $q(t,\lambda)=q_{\lambda}(t)$ is a deformation of $\dot{q}(t)\oplus v(t)$. This allows to define an inverse function of $T_{\dot{q}}\Omega(\alpha_A\circ \pi_G)$ deriving with respect to $\lambda$ at $\lambda=0$.
\end{proof}

The reduction in the $\mathfrak{LP}$ category can be now stated as follows.

\begin{theorem} \label{reduction}
Given $G$-invariant Lagrangian $L:TQ\oplus V\to \mathbb{R}$ and a curve $\dot{q}(t)\oplus v(t)$ in $\ell\Omega(Q;q_0,q_1)\oplus \Omega (V;q_0,q_1)$, the following are equivalent
\begin{enumerate} [(i)]
\item The curve $\dot{q}(t)\oplus v(t)$ is a critical point of the action $\int^{t_1}_{t_0} L(\dot{q}(t),v(t))dt$
with allowed variations $\Delta^{\ell}_{q\oplus v}$.
\item The curve $\dot{q}(t)\oplus v(t)$ satisfies the Lagrange--Poincar\'e equations \\ $\mathcal{LP}(L)(\dot{q}\oplus v)=0.$
\item The curve $$\dot{x}(t)\oplus \bar{\xi}(t) \oplus [v]_G(t)=\alpha^{TQ\oplus V}_A\circ \pi_G (\dot{q}(t)\oplus v(t))$$ in $\ell\Omega(Q/G;x_0,x_1)\oplus \Omega(\tilde{\mathfrak{g}};x_0,x_1)\oplus\Omega(V/G;x_0,x_1)$ is a critical point of the action $\int^{t_1}_{t_0} L^{(G)}(\dot{x}(t),\bar{\xi}(t),[v]_G(t))dt$
with allowed variations $\Delta^{\ell}_{x\oplus \bar{\xi}\oplus [v]_G}$.
\item The curve $\dot{x}(t)\oplus \bar{\xi} \oplus [v]_G(t)$ satisfies the Lagrange--Poincar\'e equations $\mathcal{LP}(L^{(G)})(\dot{x}\oplus \bar{\xi} \oplus [v]_G)=0.$
\end{enumerate}
\end{theorem}

\subsection{Reduction by Stages in the $\mathfrak{LP}$ category}
Next, we specify how an isomorphism in the category $\mathfrak{LP}$ induces an equivalence of variational principles. Afterwards we shall see that this implies that reduction by stages is equivalent to direct reduction in the whole  $\mathfrak{LP}$ category.
\begin{proposition}\label{isomorphism}
Let $TQ_1\oplus V_1$ and $TQ_2\oplus V_2$ be Lagrange--Poincar\'e bundles and $f:TQ_1\oplus V_1 \to TQ_2\oplus V_2$ an isomorphism in the $\mathfrak{LP}$ category. Let $L_1$ and $L_2$ be Lagrangians defined respectively in these bundles such that $L_1=L_2\circ f$. Then a curve $\dot{q}_1\oplus v_1$ satisfies the variational principle for $L_1$ if and only if $f(\dot{q}_1\oplus v_1)$ satisfies the variational principle for $L_2$.
\end{proposition}
\begin{proof}
It is enough to see that $f$ induces a bijection between the sets of curves $$\ell\Omega(Q_1;q_0,q_1)\oplus \Omega(V_1;q_0,q_1)\text{ and }\ell\Omega(Q_2;f_0(q_0),f_0(q_1))\oplus \Omega(V_2;f_0(q_0),f_0(q_1)),$$ and between the sets of variations $\Delta^{\ell}_{q\oplus v}$ and $\Delta^{\ell}_{f_0(q)\oplus f(v)}$.

On one hand, observe that $$f(\dot{q}\oplus v)=f(\dot{q})\oplus f(v)=Tf_0(\dot{q})\oplus f(v).$$ Since both, $Tf_0(\dot{q})$ and $f(v)$, project to $f_0(q)$, we conclude that $f(\dot{q}\oplus v)$ is an allowed curve in  $TQ_2\oplus V_2$ and $$\ell\Omega(Q_1;q_0,q_1)\oplus \Omega(V_1;q_0,q_1)\hookrightarrow\ell\Omega(Q_2;f_0(q_0),f_0(q_1))\oplus \Omega(V_2;f_0(q_0),f_0(q_1)).$$

On the other hand, let $\delta \dot{q}\oplus \delta v\in \Delta^{\ell}_{q\oplus v}$, that is, $$\delta v= \frac{D_1w}{dt}+[v,w]_1-\omega_{1,q}(\dot{q},\delta q) $$ and let $q(t,\lambda)\oplus v(t,\lambda)$ be a deformation producing this variation, we write
$$\left.\frac{d}{d\lambda}\right|_{\lambda =0}f(\dot{q}(t,\lambda))=Tf_0\left( \left.\frac{d}{d\lambda}\right|_{\lambda =0}\dot{q}(t,\lambda)\right)=Tf_0(\delta q)=f(\delta q),$$
\begin{equation*}
\begin{split}
\left.\frac{D_2}{d\lambda}\right|_{\lambda =0}f(v(t,\lambda))&=f\left( \left.\frac{D_1}{d\lambda}\right|_{\lambda =0}v(t,\lambda)\right)=f\left( \frac{D_1w}{dt}+[v,w]_1-\omega_{1,q}(\dot{q},\delta q) \right)\\&=\frac{D_2(f(w))}{dt}+[f(v),f(w)]_2-\omega_{2,f(q)}(f(\dot{q}),f(\delta q)),
\end{split}
\end{equation*}
where we have used that $f$ commutes with the additional structures of $TQ_1\oplus V_1$ and $TQ_2\oplus V_2$. The variation found lies in $\Delta^{\ell}_{f_0(q)\oplus f(v)}$ and, consequently, $$\Delta^{\ell}_{q\oplus v}\hookrightarrow\Delta^{\ell}_{f_0(q)\oplus f(v)}.$$

Finally, as $f$ is an isomorphism, the opposite injections are obtained analogously from $f^{-1}$.
\end{proof}

We now suppose that $L:TQ\oplus V$ is a Lagrangian invariant by the action of $G$ in $TQ\oplus V$, $N$ is a normal subgroup of $G$, and $K=G/N$ is the quotient group. Since the $\mathfrak{LP}$ category is closed under reduction, it is possible to reduce $L$ by the group  $N$ and afterwards reduce by the group $K$. The natural question is whether this is equivalent to directly reduce by $G$ or not. Let $A_N$ be a principal connection on $Q\to Q/N$, $A_{G/N}$ be a principal connection on $Q/N\to (Q/N)/(G/N)$, and $A_G$ be a principal connection on $Q\to Q/G$. These connections are said to be compatible if for all $v_q\in T_qQ$ and all $q\in Q$
$$A_G(v_q)=0 \Leftrightarrow A_N(v_q)=0 \text{ and } A_{G/N}(T\pi_N(v_q))=0.$$

%Now that we are able to reduce Lagrangians for $V\neq 0,$ this result can be extended to the whole category.  In fact, proposition \ref{isomorphism} provides a straightforward proof. It is proved in
\begin{proposition}
\cite{CMR} (Section 6.3) If the connections $A_G$, $A_N$ and $A_{G/N}$ are compatible, the map
$$\beta^{TQ\oplus V}_{(A_N,A_{G/N},A_G)}=\alpha^{T(Q/N)\oplus\tilde{\mathfrak{n}}\oplus(V/N)}_{A_{G/N}} \circ [\alpha^{TQ\oplus V}_{A_N}]_{G/N}  \circ i^{TQ\oplus V}_{(G/N)} \circ (\alpha^{TQ\oplus V}_{A_G})^{-1},$$
where $i^{TQ\oplus V}_{(G/N)}:(TQ\oplus V)/G\to((TQ\oplus V)/N)/(G/N)$ denotes the natural identification, is a $\mathfrak{LP}$ isomorphism from $T(Q/G)\oplus\tilde{\mathfrak{g}}\oplus(V/G)$ onto $$T((Q/N)/(G/N))\oplus\tilde{\mathfrak{k}}\oplus(\tilde{\mathfrak{n}}\oplus(V/N))/(G/N),$$
where $\mathfrak{n}$ is the Lie algebra of $N$ and $\mathfrak{k}$ is the Lie algebra of $K$.
\end{proposition}

Clearly,
$$L^{(G)}=(L^{(N)})^{(K)}\circ \beta^{TQ\oplus V}_{(A_N,A_{G/N},A_G)}$$ and Proposition \ref{isomorphism} concludes that $L^{(G)}$ and $(L^{(N)})^{(K)}$ pose equivalent problems. More accurately,

% dir y $A_N,A_{G/N},A_G$ conexiones como en la proposición \ref{concompatibles}, entonces la curva $$x(t)\oplus\xi(t)\oplus [v(t)]_G$$ satisface el principio variacional para $L^{(G)}:(TQ\oplus V)/G\to \R$ si y sólo si la curva $$z(t)\oplus \bar{\kappa}(t)\oplus[\bar{\eta}(t)]_K\oplus [[v(t)]_N]_{G/N}\equiv\beta^{TQ\oplus V}_{(A_N,A_{G/N},A_G)}(x(t)\oplus\xi(t)\oplus [v(t)]_G)$$ satisface el principio variacional para $(L^{(N)})^{(K)}:((TQ\oplus V)/N)/(G/N)\to \R$.

\begin{theorem} \label{stages}
Let $N<G$ be a normal subgroup and $K=G/H$. Given $G$-invariant Lagrangian $L:TQ\oplus V\to \mathbb{R}$ and a curve $\dot{q}(t)\oplus v(t)$ in $\ell\Omega(Q;q_0,q_1)\oplus \Omega (V;q_0,q_1)$, the following are equivalent:
\begin{enumerate} [\em(i)\em]
\item The curve $$\dot{y}(t)\oplus \bar{\eta}(t) \oplus [v]_N(t)=\alpha^{TQ\oplus V}_{A_N}\circ \pi_N(\dot{q}(t)\oplus v(t))$$ in $\ell\Omega(Q/N;y_0,y_1)\oplus \Omega(\tilde{\mathfrak{n}};y_0,y_1)\oplus \Omega (V/N;y_0,y_1)$ is a critical point of the action $\int^{t_1}_{t_0} L^{(N)}(\dot{y}(t),\bar{\eta}(t),[v]_N(t))dt$
with allowed variations $\Delta^{\ell}_{y\oplus \bar{\eta}\oplus [v]_N}$.
\item The curve $\dot{y}(t)\oplus \bar{\eta}(t) \oplus [v]_N(t)$ satisfies the Lagrange--Poincar\'e equations $\mathcal{LP}(L^{(N)})(\dot{y}\oplus \bar{\eta} \oplus [v]_N)=0.$
\item The curve $$\dot{z}(t)\oplus\bar{\kappa}(t)\oplus [\bar{\eta}]_K(t)\oplus [[v]_N]_K(t)=\alpha^{T(Q/N)\oplus \tilde{\mathfrak{n}}\oplus V/N}_{A_{G/N}}\circ \pi_{G/N}(\dot{y}(t)\oplus \bar{\eta}(t) \oplus [v]_N(t))$$ in $\ell\Omega((Q/N)/K;z_0,z_1) \oplus \Omega(\tilde{\mathfrak{k}};z_0,z_1)\oplus \Omega(\tilde{\mathfrak{n}}/K;z_0,z_1)\oplus\Omega((V/N)/K;z_0,z_1)$ is a critical point of the action $$\int^{t_1}_{t_0} (L^{(N)})^{(K)}(\dot{z}(t),\bar{\kappa}(t),[\bar{\eta}]_K(t),[[v]_N]_K(t))dt$$
with allowed variations $\Delta^{\ell}_{z\oplus \kappa \oplus [\bar{\eta}]_K,[[v]_N]_K}$.
\item The curve $\dot{z}(t)\oplus \bar{\kappa}(t) \oplus [\bar{\eta}]_K(t) \oplus [[v]_N]_K(t)$ satisfies the Lagrange--Poincar\'e equations $\mathcal{LP}((L^{(N)})^{(K)})(\dot{z}\oplus \bar{\kappa} \oplus [\bar{\eta}]_K \oplus [[v]_N]_K)=0.$
\end{enumerate}
\end{theorem}

\section{Noether current and vertical equations} \label{Noethersec}
In this section, we prove that the standard Noether current is not a constant of motion for Lagrangians defined on $\mathfrak{LP}$-bundles. Yet, the drift of this current reduces to the new vertical equation appearing in each step of the reduction.

\begin{definition}
Let $L:TQ\oplus V \to \mathbb{R}$ be a Lagrangian defined on an object of the $\mathfrak{LP}$ category on which a Lie group $G$ acts. We define the Noether current as the function $J:TQ\oplus V\to \mathfrak{g}^*$
\begin{equation}
J(\dot{q}\oplus v)(\eta)=\left\langle \frac{\partial L}{\partial \dot{q}}(\dot{q}\oplus v),\eta_q^Q\right\rangle ,
\end{equation}
for any $\dot{q}\oplus v \in TQ \oplus V$ and any $\eta \in \mathfrak{g}$.
\end{definition}

\begin{proposition} \label{drift}
Let $L:TQ\oplus V\to \mathbb{R}$ be a Lagrangian invariant under the action of a Lie group $G$ in the Lagrange--Poincar\'e category, and $\dot{q}(t)\oplus v(t)$ be a curve in $TQ\oplus V$ satisfying the Lagrange--Poincar\'e equations. Then the derivative of the Noether current  along the critical curve satisfies
\begin{equation}
\label{NoetherCurrentD}
\frac{d}{dt}J(\dot{q}(t)\oplus v(t))(\eta)=-\left\langle\frac{\partial L}{\partial v}(\dot{q}(t)\oplus v(t)),\omega(\dot{q}(t),\eta_{q(t)}^Q)+\eta_{v(t)}^V\right\rangle
\end{equation}
for all $\eta \in \mathfrak{g}$.
\end{proposition}
\begin{proof}
Since $L$ is invariant, choosing $\exp(s\eta)\in G$, we have $L(\dot{q}\oplus v)=L(\exp(s\eta)\dot{q}\oplus \exp(s\eta)v)$ for all $s\in \mathbb{R}$. Differentiating, we obtain
\begin{align*}
0&=\left\langle \frac{\partial L}{\partial q},\mathrm{Hor}(\eta_{\dot{q}\oplus v}^{TQ\oplus V})\right\rangle+ \left\langle \frac{\partial L}{\partial  \dot{q}},\mathrm{Ver^1}(\eta_{\dot{q}\oplus v}^{TQ\oplus V})\right\rangle+ \left\langle \frac{\partial L}{\partial  v},\mathrm{Ver^2}(\eta_{\dot{q}\oplus v}^{TQ\oplus V})\right\rangle \\ &=
\left\langle \frac{\partial L}{\partial q},\eta_q^Q\right\rangle+ \left\langle \frac{\partial L}{\partial  \dot{q}},\eta_{\dot{q}}^{TQ}\right\rangle+ \left\langle \frac{\partial L}{\partial  v},\eta_{v}^{V}\right\rangle
\end{align*}
Then, the evolution of the Noether current along $\dot{q}(t)\oplus v(t)$ is
\begin{align*}
\frac{d}{dt}J(\dot{q}(t)\oplus v(t))(\eta)=&\frac{d}{dt}\left\langle \frac{\partial L}{\partial \dot{q}},\eta_{q(t)}^Q\right\rangle=\left\langle \frac{D}{dt}\left( \frac{\partial L}{\partial \dot{q}}\right) ,\eta_{q(t)}^Q\right\rangle+\left\langle \frac{\partial L}{\partial \dot{q}},\frac{D\eta_{q(t)}^Q}{dt}\right\rangle \\
=&\left\langle \frac{D}{dt}\left( \frac{\partial L}{\partial \dot{q}}\right) ,\eta_{q(t)}^Q\right\rangle+\left\langle \frac{\partial L}{\partial \dot{q}},\eta_{\dot{q}(t)}^{TQ}\right\rangle \\=&\left\langle \frac{D}{dt}\left( \frac{\partial L}{\partial \dot{q}}\right) ,\eta_{q(t)}^Q\right\rangle-\left\langle \frac{\partial L}{\partial q},\eta_{q(t)}^Q\right\rangle- \left\langle \frac{\partial L}{\partial  v},\eta_{v(t)}^{V}\right\rangle\\
=&-\left\langle\frac{\partial L}{\partial v}(\dot{q}(t)\oplus v(t)),\omega(\dot{q}(t),\eta_{q(t)}^Q)+\eta_{v(t)}^V\right\rangle,
\end{align*}
where it has been used that $\mathrm{Hor}(\mathcal{LP})(L)(\dot{q}(t)\oplus v(t))(\eta_{q(t)}^Q)=0$.
\end{proof}

\begin{remark}
The proof of Proposition \ref{drift} does not make use of the vertical equation $\mathrm{Ver}(\mathcal{LP})(L)(\dot{q}(t)\oplus v(t))=0$. That is, the evolution of the Noether current described aboved can be applied to any curve  $\dot{q}(t)\oplus v(t)$ satisfying only the horizontal equation, $\mathrm{Hor}(\mathcal{LP})(L)(\dot{q}(t)\oplus v(t))=0$.

\end{remark}

\begin{definition}
The Noether current defined by an invariant Lagrangian is $G$-equivariant. Hence, one can define the reduced Noether current
$$
j:T(Q/G)\oplus\tilde{\mathfrak{g}}\oplus V/G\to \tilde{\mathfrak{g}}^*
$$
as
$$
j(\dot{x},\bar{\xi},[v])=[q,J(\dot{q},v)]_G,
$$
where $(\dot{x},\bar{\xi},[v])$ is any element of $T(Q/G)\oplus \tilde{\mathfrak{g}}\oplus V/G$, and $(\dot{q},v)$ projects to $(\dot{x},\bar{\xi},[v])$  by the projection from $TQ \oplus V$ to $T(Q/G)\oplus \tilde{\mathfrak{g}}\oplus V/G$.
\end{definition}
The drift \eqref{NoetherCurrentD} of the Noether current $J$ in $TQ \oplus V$ projects to $T(Q/G)\oplus \tilde{\mathfrak{g}}\oplus (V/G)$ to the condition
\begin{equation}
\label{redNoetherCurrent}
\frac{d}{dt}j(\dot{x},\bar{\xi},[v])\bar{\eta}=-\left\langle \frac{\partial L^{(G)} l}{\partial [v]},[\omega]_G(\dot{x}\oplus\bar{\xi},\bar{\eta})+[\eta^V _{v(t)}]_G \right\rangle
\end{equation}
along a solution curve $(\dot{x},\bar{\xi},[v])$ and $\bar{\eta} = [q(t),\eta ]_G$, $\eta \in \mathfrak{g}$.
We want to relate this reduced drift to the vertical equation on $T(Q/G)\oplus \tilde{\mathfrak{g}}\oplus V/G$. The vertical equation for the reduced lagrangian $L^{(G)}$ is
\begin{equation*}
\left\langle \frac{D^{\tilde{\mathfrak{g}}}}{dt}\left( \frac{\partial L^{(G)}}{\partial \bar{\xi}}\oplus\frac{\partial L^{(G)}}{\partial [v]}\right);\bar{\eta},[u]\right\rangle  =\left\langle \mathrm{ad}^*_{\bar{\xi}\oplus[v]}\left( \frac{\partial L^{(G)}}{\partial \bar{\xi}}\oplus\frac{\partial L^{(G)}}{\partial [v]}\right);\bar{\eta},[u]\right\rangle,
\end{equation*}
where $(\bar{\eta},[u])\in \tilde{\mathfrak{g}}\oplus V/G$, can be rewritten using the explicit expressions of Theorem \ref{quotientLP}
\begin{align*}
\left\langle \frac{D}{dt}\left( \frac{\partial L^{(G)}}{\partial \bar{\xi}}\right),\bar{\eta} \right\rangle +
\left\langle \left[ \frac{D^{(A,H)}}{dt}\right]\left( \frac{\partial L^{(G)}}{\partial [v]}\right), [u] \right\rangle +\left\langle \frac{\partial L^{(G)}}{\partial [v]},[\omega]_G(\dot{x},\bar{\eta})\right\rangle
\\=\left\langle \frac{\partial L^{(G)}}{\partial \bar{\xi}},[\bar{\xi},\bar{\eta}]\right\rangle +\left\langle \frac{\partial L^{(G)}}{\partial [v]}, [\nabla^{(A,V)}]_{G,\bar{\xi}}[u]-[\nabla^{(A,V)}]_{G,\bar{\eta}}[v]-[\omega]_G(\bar{\xi},\bar{\eta})+\left[ [v],[u]\right] \right\rangle.
\end{align*}
Taking alternatively, $[u]=0$ and $\bar{\eta}=0$, the vertical Lagrange--Poincar\'e equation splits into two: A new vertical equation coming from the reduction step
\begin{equation}
\label{reducedvertical}
\left\langle \frac{D}{dt}\left( \frac{\partial L^{(G)}}{\partial \bar{\xi}}\right),\bar{\eta} \right\rangle =\left\langle \mathrm{ad^*_{\bar{\xi}}}\frac{\partial L^{(G)}}{\partial \bar{\xi}},\bar{\eta}\right\rangle
+\left\langle \frac{\partial L^{(G)}}{\partial [v]},-[\nabla^{(A,V)}]_{G,\bar{\eta}}[v]-[\omega]_G(\dot{x}\oplus\bar{\xi},\bar{\eta}) \right\rangle
\end{equation}
and an equation coming from the unreduced vertical equation in $TQ\oplus V$,
\begin{align*}
\left\langle \left[ \frac{D^{(A,H)}}{dt}\right]\left( \frac{\partial L^{(G)}}{\partial [v]}\right), [u] \right\rangle
=\left\langle \mathrm{ad^*_{[v]}}\frac{\partial L^{(G)}}{\partial [v]},[u]\right\rangle+ \left\langle \frac{\partial L^{(G)}}{\partial [v]}, [\nabla^{(A,V)}]_{G,\bar{\xi}}[u]) \right\rangle.
\end{align*}

\begin{proposition}
The evolution of the reduced current given in \eqref{redNoetherCurrent} is equivalent to the group \eqref{reducedvertical} of the vertical Lagrange--Poincar\'e equation in $T(Q/G)\oplus  \tilde{\mathfrak{g}}\oplus V/G$ defined by the reduction step.
\end{proposition}
\begin{proof}
As in section \ref{noetherTQ}, given a curve $\dot{q}(t)\oplus v(t)$ in $TQ\oplus V$ and its reduced curve $\dot{x}(t)\oplus \bar{\xi}(t)\oplus [v](t)$ in $T(Q/G)\oplus \tilde{\mathfrak{g}}\oplus V/G$, we have
\begin{align*}
\left\langle \frac{d}{dt}J(\dot{q}(t)),\eta\right\rangle =\frac{d}{dt}\left\langle j(\dot{x}(t),\bar{\xi}(t)),\bar{\eta}(t)\right\rangle
=\left\langle \frac{D}{dt}\left(  \frac{\partial L^{(G)}}{\partial \bar{\xi}}\right), \bar{\eta}(t)\right\rangle-\left\langle \mathrm{ad}^*_{\bar{\xi}(t)}\frac{\partial L^{(G)}}{\partial \bar{\xi}}, \bar{\eta}(t)\right\rangle.
\end{align*}
 It is known that for any curve $[v](t)$ in $V/G$, there exists a curve $x(t)=\tau_G([v](t))$ in $Q/G$. For a fixed $t_0$, denote $x_0=x(t_0)$ and choose $q_0\in\pi^{-1}(x_0)$. There exist a unique curve $v^h_{q_0}(t)$ in $V$ such that $\tau(v^h_{q_0}(t))=x_{q_0}^h(t)$, the horizontal lift of $x(t)$ at $q_0$, and $\pi_{V,G}(v^h_{q_0}(t))=[v](t)$. According to the definition of quotient covariant derivative on $T(Q/G)\oplus \tilde{\mathfrak{g}}\oplus V/G$,
\begin{align*}
\left[\frac{D^{(A)}}{dt}\right]_{\bar{\xi}}[v](t)&=\left[ \left.\frac{D^{(A)}}{dt}\right|_{t=t_0}\exp((t-t_0)\xi)v^h_{q_0}(t)\right]\\&=\left[ \left.\frac{D^{(A)}}{dt}\right|_{t=t_0}\exp((t-t_0)\xi)v^h_{q_0}(t_0)\right] +\left[ \left.\frac{D^{(A)}}{dt}\right|_{t=t_0}v^h_{q_0}(t)\right]\\&=[\xi_{v}^V]+\left[\frac{D^{(A,H)}}{dt}\right][v](t)
\end{align*}
Thus $\left[\frac{D^{(A,V)}}{dt}\right]_{\bar{\xi}}[v](t)=[\xi_{v}^V]$ and the drift equation
\begin{align*}
\frac{d}{dt}J(\dot{q}(t)\oplus v(t))(\eta)=-\left\langle\frac{\partial L^{(G)}}{\partial v}(\dot{q}(t)\oplus v(t)),\omega(\dot{q}(t),\eta_{q(t)}^Q)+\eta_{v(t)}^V\right\rangle
\end{align*}
reduces to
\begin{align*}
\left\langle \frac{D}{dt}\left(  \frac{\partial L^{(G)}}{\partial \bar{\xi}}\right), \bar{\eta}(t)\right\rangle-\left\langle \mathrm{ad}^*_{\bar{\xi}(t)}\frac{\partial L^{(G)}}{\partial \bar{\xi}}, \bar{\eta}(t)\right\rangle=-\left\langle \frac{\partial L^{(G)}}{\partial [v]},[\omega]_G(\dot{x}\oplus\bar{\xi},\bar{\eta})+[\nabla^{(A,V)}]_{G,\bar{\eta}}[v]\right\rangle,
\end{align*}
which is the new vertical equation obtained in the reduction process.
\end{proof}

\section{The Poisson category $\mathfrak{LP^*}$}
The objects of the category $\mathfrak{LP^*}$ are  bundles $\bar{\tau}_Q\oplus \bar{\tau}: T^*Q\oplus V^*\to Q$, such that $V^*\to Q$ is the dual of a Lie algebra vector bundle $V\to Q$, equipped with a linear connection $\nabla$, and $Q$ has a 2-form $\omega$ taking values in $V$. An element $T^*Q\oplus V^*\to Q\in \mathfrak{LP^*}$ can be thought as the dual of a vector bundle $TQ\oplus V\to Q\in \mathfrak{LP}$, with the same structures $\omega$, $[\cdot , \cdot ]$ and $\nabla$ (we use the same notation for a linear connection and its dual).

Elements of  $\mathfrak{LP^*}$  are Poisson manifolds as we now describe.

%That bundle, $TQ\oplus V$, belongs to $\mathfrak{LP}$ if its set of sections has a Lie algebra structure as described in 1.(d), which provides a Poisson bracket as follows.

\begin{proposition} \label{propPoisson}
\cite{CMR}
If $\bar{\tau}_Q\oplus \bar{\tau}: T^*Q\oplus V^*\to Q$ is dual to an element $TQ\oplus V\to Q  \in \mathfrak{LP}$, then there is a unique Poisson bracket
\[
\{,\}:C^{\infty}(T^*Q\oplus V^*)\times C^{\infty}(T^*Q\oplus V^*) \to C^{\infty}(T^*Q\oplus V^*)
\]
characterized by its restriction to affine functions as
\begin{enumerate}[\indent {}]
\item $\{\bar{f},\bar{g}\}=0$ for all $f,g\in C^{\infty}(Q)$
\item $\{\bar{f},P(X\oplus w)\}=\overline{X[f]}$
\item $\{P(X_1\oplus w_1),P(X_2\oplus w_2)\}=-P([X_1\oplus w_1,X_2\oplus w_2])$
\end{enumerate}
where
\begin{itemize}
\item for $f$ in $C^{\infty}(Q)$, we define $\bar{f}=f\circ \bar{\tau}_Q\oplus \bar{\tau}$
\item for $X\oplus w \in \Gamma (TQ\oplus V)$ we define $$P(X\oplus w)(p\oplus \nu)= \langle p, X\rangle+\langle \nu, w \rangle.$$ for all $p\oplus \nu \in T^*Q\oplus V^*$
\end{itemize}
\end{proposition}

This Poisson bracket behaves well under reduction by stages, in particular we have:

\begin{proposition}
\cite{CMR}
A $\mathfrak{LP}$ action of a Lie group on an element $TQ\oplus V\in \mathfrak{LP}$ naturally induces an action on  $T^*Q\oplus V^*\in \mathfrak{LP^*}$ such that the projection $T^*Q\oplus V^* \to T^*(Q/G)\oplus \tilde{\mathfrak{g}}^*\oplus V^*$ is a Poisson map with respect to the Poisson brackets defined in Proposition \ref{propPoisson}.
\end{proposition}

The explicit expression of the Poisson bracket of Proposition \ref{propPoisson} was given for Lagrange--Poincar\'e bundles of the type $T(Q/G)\oplus \tilde{\mathfrak{g}}$ in \cite{moscow}. The generalization to the whole $\mathfrak{LP}$ category is given hereunder:

\begin{proposition}
If $\bar{\tau}_Q\oplus \bar{\tau}: T^*Q\oplus V^*\to Q$ is an element of the $\mathfrak{LP^*}$ category dual to an element $TQ\oplus V\to Q  \in \mathfrak{LP}$, then the Poisson bracket characterized in Proposition \ref{propPoisson}
is the following
\begin{equation*}
\label{poisson_bracket}
\{f,g\}=\frac{\partial f}{\partial q}\frac{\partial g}{\partial p}-\frac{\partial g}{\partial q}\frac{\partial f}{\partial p} +\left\langle \nu,\omega\left(\frac{\partial f}{\partial p},\frac{\partial g}{\partial p}\right)\right\rangle+\left\langle \nu,\left[ \frac{\partial g}{\partial \nu},\frac{\partial f}{\partial \nu}\right] \right\rangle .
\end{equation*}
\end{proposition}

\begin{proof}
It is necessary to prove that this expression defines a Poisson bracket on $T^*Q\oplus V^*$ and that it has the required properties for affine functions. Skew-symmetry and alternativity of $\{,\}$ are a consequence of the skew-symmetry and alternativity of $\omega$ and the Lie bracket on $V$. Proving the Jacobi identity for general functions $f,g\in C^{\infty}(T^*Q\oplus V^*)$ directly from the formula can be excruciating, yet, since $\{,\}$ does only depend on the differential of $f,g$ it suffices to prove it for affine funtions.

We first prove that $\{,\}$ has the required properties for affine functions in order to coincide with the bracket above. Since for any $\bar{f}$, with $f\in C^{\infty}(Q)$, $\frac{\partial \bar{f}}{\partial p}=\frac{\partial \bar{f}}{\partial \nu}=0$ then $\{\bar{f},\bar{g}\}=0$ for all $f,g\in C^{\infty}(Q)$. For any $P(X\oplus w)$, $\frac{\partial P(X\oplus w)}{\partial p}=X$ and $\frac{\partial P(X\oplus w)}{\partial \nu}=w$. Thus,
$$\{\bar{f},P(X\oplus w)\}=\frac{\partial f}{\partial q}\frac{\partial P(X\oplus w)}{\partial p}=X[\bar{f}]=\overline{X[f]}.$$ Finally, to obtain $\{P(X_1\oplus w_1),P(X_2\oplus w_2)\}$ is necessary to calculate $\frac{\partial P(X\oplus w)}{\partial q}$. This can be done using local coordinates in $T^*Q\oplus V^*$, $P(X\oplus w)=X^i(q^1\dots q^n)p_i+w^{\alpha}(q^1\dots q^n)\nu_{\alpha}$ where $(q^i,p_i), i=1\dots n=\mathrm{dim}Q$ are coordinates in $T^*Q$ and $\nu_{\alpha}, \alpha=1\dots m$ are independent local sections on $V$. Then, $$\frac{\partial P(X\oplus w)}{\partial q^j}=\frac{\partial X^i}{\partial q^j}p_i+\left( \frac{\partial w^{\alpha}}{\partial q^j}+\Gamma^{\alpha}_{\beta j}w^{\beta}\right) \nu_{\alpha}.$$ From where it follows that $$\left\langle \frac{\partial P(X_1\oplus w_1)}{\partial q},X_2\right\rangle = P(X_1\circ X_2)+ P(\nabla_{X_2}w_1).$$ Hence,
\begin{align*}
&\{P(X_1\oplus w_1),P(X_2\oplus w_2)\}\\&=\left\langle \frac{\partial P(X_1\oplus w_1)}{\partial q},X_2\right\rangle-\left\langle \frac{\partial P(X_2\oplus w_2)}{\partial q},X_1\right\rangle+\left\langle \nu, \omega(X_1,X_2)-[w_1,w_2]\right\rangle\\&=-P([X_1,X_2])(p,\nu)+P(\nabla_{X_2}w_1-\nabla_{X_1}w_2)(p,\nu)+P(\omega(X_1,X_2)-[w_1,w_2])(p,\nu)\\&=-P([X_1\oplus w_1,X_2\oplus w_2])
\end{align*}

These properties imply the Jacobi identity for affine functions, and hence, for all. In fact
$$\{\bar{f},\{\bar{g},\bar{h}\}\}+\{\bar{g},\{\bar{h},\bar{f}\}\}+\{\bar{h},\{\bar{f},\bar{g}\}\}=0+0+0=0;$$
\begin{align*}
\{\bar{f},\{\bar{g},P(X\oplus w)\}\}+\{\bar{g},\{P(X\oplus w),\bar{f}\}\}+\{P(X\oplus w),\{\bar{f},\bar{g}\}\}\\=\{\bar{f},\overline{X[g]}\}+\{\bar{g},-\overline{X[f]}\}+0=0;
\end{align*}
\begin{align*}
&\{\bar{f},\{P(X_1\oplus w_1),P(X_2\oplus w_2)\}\}\\&+\{P(X_1\oplus w_1),\{P(X_2\oplus w_2),\bar{f}\}\}+\{P(X_2\oplus w_2),\{\bar{f},P(X_1\oplus w_1)\}\}\\
=&\{\bar{f},-P([X_1\oplus w_1,X_2\oplus w_2])\}+\{P(X_1\oplus w_1),-\overline{X_2[f]}\}+ \{P(X_2\oplus w_2),\overline{X_1[f]}\}\\=&-\overline{[X_1,X_2][f]}+\overline{X_1[X_2[f]]}-\overline{X_2[X_1[f]]}=0;
\end{align*}
For three functions $P(X_i\oplus w_i)$, the Jacobi identity is based on the Jacobi identity for the Lie bracket on $\Gamma(TQ\oplus V)$
\begin{align*}
\{P(X_1\oplus w_1),\{P(X_2\oplus w_2),P(X_3\oplus w_3)\}\}\\=-\{P(X_1\oplus w_1),P([X_2\oplus w_2,X_3\oplus w_3])\}=P([X_1\oplus w_1,[X_2\oplus w_2,X_3\oplus w_3]])
\end{align*}
At last, the Leibniz identity is easily obtained from the expression.
\end{proof}

From this last result we can give an intrinsic definition of the $\mathfrak {LP}^*$ without an explicit notion of duality with respect to the Lagrange--Poincar\'e category. More precisely:

\begin{definition}
The objects of $\mathfrak{LP}^*$ are bundles $T^*Q\oplus V^*\to Q$, such that $V^*\to Q$ is the dual of a Lie algebra vector bundle $V\to Q$, equipped with a linear connection $\nabla$, $Q$ has a 2-form $\omega$ taking values in $V$, and the bracket given in \eqref{poisson_bracket} is a Poisson bracket.
\end{definition}

From this definition, it is not difficult to see that $T^*Q\oplus V^*\to Q\in \mathfrak{LP}^*$ if and only if $TQ\oplus V \in \mathfrak{LP}$ with the same structures.

We finally study the dynamical equations defined by the Poisson bracket in the dual Lagrange--Poincar\'e  category. Let $H:T^*Q\oplus V^*\to \mathbb{R}$ be a Hamilltonian on an element of $\mathfrak{LP}^*$. A simple computation shows that the  Hamiltonian field defined by the bracket \eqref{poisson_bracket} is given by
$$X_H=\left( \frac{\partial H}{\partial p},-\frac{\partial H}{\partial q}+\left\langle \nu, \omega\left( \cdot,\frac{\partial H}{\partial p}\right) \right\rangle, \mathrm{ad}^*_{\frac{\partial H}{\partial \nu}}\nu \right) $$
A curve $(p(t),\nu(t))$ in $T^*Q\oplus V^*$, projecting to a curve $q(t)$, is an integral of $X_H$ if and only if
\begin{eqnarray*}
\dot{q}&=&\frac{\partial H}{\partial p}\\
\frac{Dp}{dt}&=&-\frac{\partial H}{\partial q}+\left\langle \nu, \omega\left( \cdot,\frac{\partial H}{\partial p}\right) \right\rangle\\
\frac{\nabla \nu}{dt}&=&\mathrm{ad}^*_{\frac{\partial H}{\partial \nu}}\nu.
\end{eqnarray*}
These equations are called the Hamilton-Poincar\'e equations for $H$.

Given a Lagrangian $L:TQ\oplus V\to \mathbb{R}$ in the Lagrange--Poincar\'e category, we define the Legendre map
$\mathbb{F}L:TQ\oplus V \to T^*Q\oplus V^*$  in the usual way as the fiber derivative of $L$. We write this as
\[
(q,\dot{q},v)\mapsto (q,p=\partial L/\partial \dot{q}, \nu = \partial L /\partial v).
\]
If $\mathbb{F}L$ is a diffeomorphism, we define the Hamiltonian $H:T^*Q\oplus V^*\to \mathbb{R}$ as
\begin{align*}
H:T^*Q\oplus V^*\to &\mathbb{R} \\
(p,\nu) \mapsto & \langle p,\dot{q}\rangle + \langle \nu,v\rangle- L(\dot{q}\oplus v),
\end{align*}
where $(\dot{q},v)=\mathbb{F}L^{-1}(p,\nu)$. Then  Lagrange--Poincar\'e equations and Hamilton-Poincar\'e equations are equivalent.
Reduction using momentum techniques as in \cite{RedstagesPoisson}, that is, tracking the symplectic leaves structure,  should be a subject of future work. Specially, under the light of results in Section \ref{Noethersec}.

\section{Examples}
\subsection{Examples outside of $\mathfrak{RI}$}

Let $Z$ be an abelian Lie group, $\mathfrak{z}$ be its abelian Lie algebra, and $P\to M$ be a $Z$-principal bundle. Since the adjoint action is trivial, the adjoint bundle $\mathfrak{\tilde{z}}\to M$ is a trivial Lie algebra vector bundle. Let $V \to M$ be a non-trivial vector bundle equipped with a trivial fiberwise Lie bracket. Then the bundle $TM\oplus V$ cannot belong to the subcategory $\mathfrak{RI}$ of reduced tangent bundles. However, it can be seen as an object in $\mathfrak{LP}$, considering the adequate triple $[,]$, $\omega$ and $\nabla$. Firstly, we choose a flat connection $\nabla$ on $V$. There are instances of non-trivial vector bundles with flat connections, all of them over non-simply connected manifolds. As $\nabla$ is flat and the Lie bracket on the fibers is trivial, conditions $1.(e')$ and $1.(f')$ are clearly satisfied. With respect to  $1.(d')$, we take any closed 2-form $\omega$ with respect to the covariant derivative defined by $\nabla$ (that is, a representative of the cohomology with values in $V$ defined by the covariant differential). In particular, we can even choose $\omega =0$. In short, the problems posed in $TM\oplus V$ with trivial $[,]$, flat connection $\nabla$ and a form $\omega$ chosen as above are set in $\mathfrak{LP}$ and not in the subcategory $\mathfrak{RI}$.

We specify an example within this context. Let $L: TM\oplus V \to \mathbb{R}$ be the Lagrangian given by
\[
L(q,\dot{q},v)=g(\dot{q}, \dot{q}) + h(v,v),
\]
where $g$ is a Riemannian metric on $M$ and $h$ is a vector bundle metric on $V$. The Lagrange--Poincar\'e equations are
\begin{eqnarray*}
\frac{\nabla \dot{q}}{dt}&=& h(v, \omega (\dot{q},\cdot)),\\
\frac{\nabla v}{dt}&=&0.
\end{eqnarray*}
The first equation provides the Newtonian dynamics of a particle on $M$ under a force defined by $v$ and $\omega$. This equation reduces to the simple geodesic equation when $\omega =0$. The second equation is just the parallel transport of $v$ along $q(t)$.

A similar construction can be given in a wider class of fiber bundles. Let $G$ be any non-abelian Lie group, $\mathfrak{g}$ be its Lie algebra, and $\mathfrak{z}$ be the center of $\mathfrak{g}$. If $P\to M$ is a $G$-principal bundle, then the adjoint bundle $\mathfrak{\tilde{g}}\to M$ has a trivial subbundle with fiber dimension equal to the dimension of the center. Indeed, the subbundle
\[
\mathfrak{\tilde{z}}=\{ [p,B]_G: p\in P, B\in \mathfrak{z}\}
\]
can be identified with $M\times \mathfrak{z}$ as the adjoint action in the center is trivial. In addition, at each point, the subbundle is the center of the Lie algebra on the fibers of $\mathfrak{\tilde{g}}$. Furthermore, suppose that $\mathfrak{g}$ is a reductive Lie algebra, that is, $\mathfrak{g}=\mathfrak{z}\oplus \mathfrak{s}$ where $\mathfrak{s}$ is semisimple. Accordingly, the adjoint bundle can be decomposed as
$\mathfrak{\tilde{g}}= \mathfrak{\tilde{z}}\oplus \mathfrak{\tilde{s}}$. If we replace $\mathfrak{\tilde{z}}$ by a non-trivial vector bundle $V\to M$ equipped with a trivial bracket, the bundle $TQ\oplus V\oplus\tilde{\mathfrak{s}}\to M$ does not belong to the subcategory $\mathfrak{RI}$ even though it is an element of $\mathfrak{LP}$ with a convenient choice of connection $\nabla$ and 2-form $\omega$.

\subsection{Lagrangian Depending on a Parameter}

Let $L:T(G\times Q)\times V^*\to \mathbb{R}$ be a Lagrangian function where $G$ is a Lie group, $Q$ is a manifold, and $V^*$ is the dual of a vector space $V$, for which the variable $ V^*$ is understood as a  parameter. More precisely, for each $a_0 \in V^*$, we are looking for curves $(g(t),q(t),a_0)$ in $G\times Q \times V^*$ which are critical points of the action
\begin{equation*}
\int_{t_0}^{t_1} L([g]^{(1)},[q]^{(1)},a_0) dt
\end{equation*}
with restrictions on variations given by $\delta g(t_i)=\delta q(t_i)=0$ for $i=0,1$ and $a_0$ fixed ($\delta a_0=0$). For the sake of simplicity, hereafter $[g]^{(1)}$ and $[q]^{(1)}$ will be respectively denoted $(g,\dot{g})$ and $(q,\dot{q})$. We note that we are in a similar setting to the semi-direct product Lagrangian in \cite{HMR98A}.

We consider a representation of $G$ on $V$ as well as the induced dual representation on $V^*$, so that $G$ acts on $T(G\times Q)\times V^*$ as follows
$$h(g,q,\dot{g},\dot{q},a_0)=(hg,q,h\dot{g},\dot{q},ha_0)$$
for all $h\in G$ and $(g,q,\dot{g},\dot{q},a_0)\in T(G\times Q)\times V^*$. We assume that $L$ is invariant with respect to that action and we define a reduced Lagrangian
\begin{equation*}
\begin{split}
l:\mathfrak{g}\times TQ\times V^*\to & \mathbb{R} \\
(\xi,q,\dot{q},a)\mapsto & l(\xi,q,\dot{q},a)=L(e,q,\xi,\dot{q},a)
\end{split}
\end{equation*}
Therefore, for all $g\in G$, all $q\in Q$ and all $a_0\in V^*$,
$$L(g,q,\dot{g},\dot{q},a_0)=l(\xi,q,\dot{q},a),$$
where $\xi=g^{-1}\dot{g}$ and $a=g^{-1}a_0$.

The reduction of the variational problem of $L$ defined above is as follows. Solutions of the reduced problem are curves $(\xi(t), q(t),a(t))$ in $\mathfrak{g}\times TQ\times V^*$, whith $\xi(t)=g^{-1}(t)\dot{g}(t)$ and $a(t)=g^{-1}(t)a_0$, that are critical elements for the action defined by $l$ with restricted variations
$$\delta\xi=\dot{\eta}+[\xi,\eta]$$ where $\eta$ is a curve in $\mathfrak{g}$ such that $\delta \eta(t_i)=0$ for $i=0,1$;
$\delta q(t_i)=0,$ for $i=0,1$; and
$$\delta a=- \eta_a^{V^*},$$ satisfying the additional condition \begin{equation}\label{paramred}
\dot{a}+ \eta_a^{V^*}=0
\end{equation} coming from $\dot{a_0}=0$. This equivalence of principles inmediately leads to equations
\begin{equation}\label{VerLPV}
-\frac{D}{dt}\left(\frac{\partial l}{\partial \xi}\right)+\text{ad}^*_{\xi}\frac{\partial l}{\partial \xi}+\frac{\partial l}{\partial a}\diamond a=0,
\end{equation}
\begin{equation}\label{HorLPV}
 \frac{\partial l}{\partial q}-\frac{D}{dt}\left(\frac{\partial l}{\partial \dot{q}}\right)=0,
\end{equation}
where we define
$(b\diamond a) (\eta)=-\langle \eta_a^{V^*}, b\rangle,$
for all $\eta\in\mathfrak{g},a\in V^*$, and $b\in V$. Thus, these Lagrange--Poincar\'e equations together with condition, $\dot{a}+\eta^{V^*}_a=0$ solve, directly from the variational principle, the problem of an invariant Lagrangian depending on a parameter.

However, it is interesting to obtain these equations from different perspectives. For example, in \cite{HMR98A}, Euler-Poincar\'e reduction is performed when $Q$ is a point and $L:TG \times V^*\to \mathbb{R}$.  In general, the equations can be also obtained combining Lagrange--Poincar\'e reduction and Lagrange multiplyers (see \cite{CMR} and \cite{CM87}) for Lagrangians
\begin{equation*}
\begin{split}
L^V:T(G\times Q\times V^* \times V)\to &\mathbb{R}\\
(g,q,a,b,\dot{g},\dot{q},\dot{a},\dot{b})\mapsto & L(g,q,\dot{g},\dot{q},a)+\langle\dot{a}+g^{-1}\dot{g}a,b \rangle .
\end{split}
\end{equation*}
We are going to give here a different approach within the Lagrange--Poincar\'e category $\mathfrak{LP}$.

For that, we define the Lagrangian
\begin{equation*}
\begin{split}
\bar{L}:T(G\times Q) \oplus \tilde{V}^* \oplus \tilde{V}\to &\mathbb{R}\\
(g,q,\dot{g},\dot{q},a_0,b_0)\mapsto & L(g,q,\dot{g},\dot{q},a_0)+\langle a_0,b_0 \rangle,
\end{split}
\end{equation*}
where $\tilde{V}=(G\times Q)\times V$ is a trivial vector bundle endowed with the correspondingly trivial connection, $\nabla$; and a trivial Lie bracket, $[\cdot,\cdot]=0$, in the fibers. Furthermore, $\tilde{V}^*$ accounts for the dual of this bundle similarly equipped with a trivial Lie bracket. This, together with the null $\tilde{V}^* \oplus \tilde{V}$-valued $2$-form on $G\times Q$ makes $T(G\times Q) \oplus \tilde{V}^* \oplus \tilde{V}$ a Lagrange--Poincar\'e bundle and, hence, there is a notion of Lagrange--Poincar\'e equations for $\bar{L}$.

Since the $2$-form of this $\mathfrak{LP}$-bundle is zero, the horizontal Lagrange--Poincar\'e equation of $\bar{L}$ depends only on horizontal derivatives of $\bar{L}$. These coincide with the horizontal derivatives of $L$, and consequently, the horizontal Lagrange--Poincar\'e equation of $\bar{L}$ coincide with the Euler--Lagrange equations for the original Lagrangian $L$. On the other part, there are two vertical Lagrange--Poincar\'e equations as $\tilde{V}^* \oplus \tilde{V}$ has two factors. The vertical equation coming from $\tilde{V}$ imposes that $a_0$ is a fixed parameter
\begin{equation*}
0=\frac{D}{dt}\left(\frac{\partial \bar{L}}{\partial b_0}\right)=\dot{a}_0,
\end{equation*}
while the vertical Lagrange--Poincar\'e equation coming from $\tilde{V}^*$ gives the evolution of the auxiliary variable $b_0$:
\begin{equation*}
0=\frac{D}{dt}\left(\frac{\partial \bar{L}}{\partial a_0}\right)=\frac{D}{dt}\left(\frac{\partial L}{\partial a_0}\right)+\dot{b}_0,
\end{equation*}

To reduce the Lagrange--Poincar\'e equations of $\bar{L}$ to equations \eqref{paramred}, \eqref{VerLPV}, and \eqref{HorLPV}, we first discuss the action of $G$ on $T(G\times Q) \oplus \tilde{V}^* \oplus \tilde{V}$ and the resulting quotient $\mathfrak{LP}$-bundle. Since $T(G\times Q) \oplus \tilde{V}^* \oplus \tilde{V}\cong T(G\times Q) \oplus (G\times Q\times V^*) \oplus (G\times Q\times V)$ the action of $h\in G$ can be explicited as
$$h\cdot\left( (g,q,\dot{g},\dot{q})\oplus(g,q,b_0)\oplus(g,q,a_0)\right) =(hg,q,h\dot{g},\dot{q})\oplus(hg,q,hb_0)\oplus(hg,q,ha_0).$$ On the other hand, we have the natural isomorphism for the quotient bundle
\begin{align*}
TQ \oplus \tilde{\mathfrak{g}} \oplus \tilde{V}^* /G \oplus \tilde{V} /G \hspace{2mm} \cong \hspace{2mm}& TQ \oplus (Q\times \mathfrak{g}) \oplus (Q\times V^*) \oplus (Q\times V) \\
(q,\dot{q})\oplus[(e,q),\xi]_G\oplus[g,q,b_0]_G\oplus[g,q,a_0]_G \longleftrightarrow & (q,\dot{q})\oplus(q,\xi)\oplus(q,b)\oplus (q,a),
\end{align*}
where $b=g^{-1}b_0$, and $a=g^{-1}a_0$.

The bundle $G \times Q \to Q$ is equipped with the trivial connection that induces in the adjoint bundle $\tilde{\mathfrak{g}} \cong Q\times \mathfrak{g}$ the trivial covariant derivative
$$\frac{D}{dt}(q(t),\xi(t))=(q(t),\dot{\xi}(t)).$$ This connection is flat, that is, $\tilde{B}=0$, and $Q\times \mathfrak{g}$ has a fiber-wise Lie bracket given by $[(q,\xi_1),(q,\xi_2)]=(q,[\xi_1,\xi_2])$. The factor $(Q\times V^*) \oplus (Q\times V)\cong Q\times V\times V^*$ has null Lie bracket and $(Q\times V\times V^*)$-valued 2 form on $Q$ coming from the respective structures in $\tilde{V}^* \oplus \tilde{V}$. However, reducing the trivial connection of $\tilde{V}^* \oplus \tilde{V}$,
$$\frac{D}{dt}(g(t),q(t),a_0(t),b_0(t))=(g(t),q(t),\dot{a}_0(t),\dot{b}_0(t)).$$
to a connection in $\tilde{V}^* /G \oplus \tilde{V} /G \cong (Q\times V^*) \oplus (Q\times V)$ is somewhat trickier: It requires to separate the horizontal and vertical component as explained in section \ref{Sect.LPcategory}. The explicit calculation of the vertical component of the covariant derivative of a curve $v(t)=(g(t),q(t),a_0(t),b_0(t))$ in $\tilde{V}^* \oplus \tilde{V}$ is done as in \cite{CMR}. For a fixed $t_0$, $\tau(v(t))=(g(t),q(t))=h(t)(g_0,q(t))$ where $g_0=g(t_0)$ and $h(t)=g(t)g_0^{-1}$, then the horizontal component of $Dv(t)/dt$ is
\begin{align*}
\left. \frac{D^H}{dt}\right|_{t=t_0}v(t)=\left. \frac{D}{dt}\right|_{t=t_0}(h(t)^{-1}v(t))=\left. \frac{D}{dt}\right|_{t=t_0}(g_0,q(t),g_0a(t),g_0b(t))=\\
=(g_0,q(t),g_0\dot{a}(t_0),g_0\dot{b}(t_0))=(g(t),q(t),g(t)\dot{a}(t),g(t)\dot{b}(t))_{\vert t=t_0}.
\end{align*}
Differenciating $a=g^{-1}a_0$, we obtain $\dot{a}+\xi^{V^*}_a=g^{-1}\dot{a_0}$, and consequently,
\begin{align*}
\frac{D^V}{dt}v(t)=\frac{D}{dt}v(t)-\frac{D^H}{dt}v(t)&=
(g(t),q(t),\dot{a_0}(t)-g(t)\dot{a}(t),\dot{b_0}(t)-g(t)\dot{b}(t))\\&=(g(t),q(t),g(t)\xi^V_{a(t)},g(t)\xi^{V^*}_{b(t)}).
\end{align*}
It follows that for a section $[v]_G=(q,a(q),b(q))$ of $Q\times V\times V^*\to Q$, the horizontal quotient connection is
$$\left[ \nabla^{(H)}\right] _{G,(q,\dot{q},\xi)}[v]_G=(q,\dot{q}[a],\dot{q}[b]),$$
and the vertical quotient connection is
$$\left[ \nabla^{(V)}\right] _{G,(q,\dot{q},\xi)}[v]_G=(q,\xi^V_{a(q)},\xi^V_{b(q)}).$$
Direct application of Theorem \ref{quotientLP} gives the $\mathfrak{LP}$-bundle structure on the quotient bundle, $$TQ \oplus (Q\times \mathfrak{g}) \oplus (Q\times V^*) \oplus (Q\times V)\cong TQ \oplus (Q\times \mathfrak{g} \times V^*\times V));$$
determined by
$$\nabla^{\tilde{\mathfrak{g}}}_{(q,\dot{q})}(q,\xi, a, b)=\nabla^{A}_{(q,\dot{q})}(q,\xi)\oplus\left[\nabla^{(H)}\right]_{G,(q,\dot{q})}(q,a,b)=(q,\dot{q}[\xi], \dot{q}[a], \dot{q}[b]);$$
$$\omega^{\tilde{\mathfrak{g}}}=\tilde{B}\oplus [\omega]_G=0;$$
$$[(q,\xi_1, a_1, b_1),(q,\xi_2, a_2, b_2)]^{\tilde{\mathfrak{g}}}=(q,[\xi_1,\xi_2],(\xi_1)^{V^*}_{a_2}-(\xi_2)^{V^*}_{a_1},(\xi_1)^V_{b_2}-(\xi_2)^V_{b_1}).$$

Finally, since $L$ is $G$-invariant and $\langle ga_0,gb_0\rangle=\langle a_0,b_0\rangle$, $\bar{L}$ is $G$-invariant and the reduced Lagrangian is
\begin{align*}
\bar{l}:TQ \oplus (Q\times \mathfrak{g} \times V^*\times V))\to &\mathbb{R}\\
(q,\dot{q},\xi,a,b)\mapsto & l(q,\dot{q},\xi,a)+\langle a,b \rangle.
\end{align*}
Its horizontal Lagrange--Poincar\'e equation is
\begin{equation*}
0=\frac{D}{dt}\left(\frac{\partial \bar{l}}{\partial \dot{q}}\right)-\frac{\partial \bar{l}}{\partial q}=\frac{D}{dt}\left(\frac{\partial l}{\partial \dot{q}}\right)-\frac{\partial l}{\partial q},
\end{equation*}
which coincides with equation \eqref{HorLPV}, while its vertical Lagrange--Poincar\'e equation is
\begin{equation*}0=-\frac{D}{dt}\left(\frac{\partial \bar{l}}{\partial (\xi,a,b)}\right)+\text{ad}_{(\xi,a,b)}\left(\frac{\partial \bar{l}}{\partial (\xi,a,b)}\right)
\end{equation*}
Applying this expression to any variation $(\delta\xi,\delta a,\delta b)\in \mathfrak{g}\times V^*\times V$ we have
\begin{align*}
0=&\left\langle -\frac{D}{dt}\left(\frac{\partial \bar{l}}{\partial (\xi,a,b)}\right)+\text{ad}_{(\xi,a,b)}\left(\frac{\partial \bar{l}}{\partial (\xi,a,b)}\right);\delta\xi,\delta a,\delta b\right\rangle \\
=& -\left\langle\frac{D}{dt}\left(\frac{\partial \bar{l}}{\partial \xi}\right),\delta\xi\right\rangle -\left\langle \frac{D}{dt}\left(\frac{\partial \bar{l}}{\partial a}\right),\delta a\right\rangle -\left\langle \frac{D}{dt}\left(\frac{\partial \bar{l}}{\partial b}\right),\delta b\right\rangle \\
&+\left\langle \frac{\partial \bar{l}}{\partial (\xi,a,b)};[\xi,\delta\xi],(\xi)^{V^*}_{\delta a}-(\delta\xi)^{V^*}_{a},(\xi)^V_{\delta b}-(\delta \xi)^V_{b}\right\rangle \\
=&-\left\langle\frac{D}{dt}\left(\frac{\partial l}{\partial \xi}\right),\delta\xi\right\rangle-\left\langle \frac{D}{dt}\left(\frac{\partial l}{\partial a}\right),\delta a\right\rangle-\left\langle \dot{b},\delta a\right\rangle-\left\langle \frac{D}{dt}\left(\frac{\partial l}{\partial b}\right),\delta b\right\rangle \\
&+\left\langle \frac{\partial l}{\partial \xi},[\xi,\delta\xi]\right\rangle +\left\langle \frac{\partial l}{\partial a},(\xi)^{V^*}_{\delta a}-(\delta\xi)^{V^*}_{a}\right\rangle \\
&+\left\langle b,(\xi)^{V^*}_{\delta a}-(\delta\xi)^{V^*}_{a}\right\rangle +\left\langle \frac{\partial l}{\partial b},(\xi)^V_{\delta b}-(\delta \xi)^V_{b}\right\rangle \\
=&\left\langle -\frac{D}{dt}\left(\frac{\partial l}{\partial \xi}\right)+\text{ad}_{\xi}\left(\frac{\partial l}{\partial \xi}\right)-\frac{\partial l}{\partial a}\diamond a,\delta\xi\right\rangle \\
&+\left\langle -\frac{D}{dt}\left(\frac{\partial l}{\partial a}\right)-\dot{b}-\xi_{\frac{\partial l}{\partial a}}^V-\xi_{b}^V,\delta a\right\rangle -\left\langle \dot{a}+\xi_{a}^V,\delta b\right\rangle
\end{align*}
where it has been repeatedly used that $\langle a,\xi^V_b\rangle+\langle\xi^{V^*}_a,b\rangle=0$ for all $\xi\in \mathfrak{g}$ and the abuse of notation $\langle a,b\rangle=\langle b,a\rangle$. As the variations $\delta\xi,\delta a,\delta b$ are free, we recover \eqref{paramred}, \eqref{VerLPV} as well as the evolution of the auxiliary variable $b$.

\end{document}